\documentclass[onefignum,onetabnum]{siamart250211}

\usepackage{amsmath,amssymb,amsfonts}

\usepackage{hyperref}
\usepackage{graphicx}
\usepackage{epstopdf}
\usepackage{mathtools}

\ifpdf
  \DeclareGraphicsExtensions{.eps,.pdf,.png,.jpg}
\else
  \DeclareGraphicsExtensions{.eps}
\fi

\usepackage[top=3.375cm,
            bottom=3.375cm,
            inner=4.275cm,
            outer=4.275cm]{geometry}

\newsiamremark{remark}{Remark}
\newsiamremark{example}{Example}

\headers{}{D. S. Winterrose}

\title{High-Order $L^2$-Galerkin Schemes \\ for the Decoupled Potential Integral Equations}

\author{David Scott Winterrose \thanks{DTU, Kgs. Lyngby (\email{dswint@dtu.dk}).}}

\usepackage{subcaption}
\usepackage{caption}
\usepackage{amsopn}
\usepackage{mathrsfs}
\usepackage{pgfplots}
\pgfplotsset{compat=1.18}

\renewcommand{\figurename}{Figure}

\ifpdf
\hypersetup{
  pdftitle = {High-Order Galerkin DPIE},
  pdfauthor = {David Scott Winterrose}
}
\fi

\begin{document}

\maketitle

\begin{abstract}
We prove Sobolev space well-posedness of the decoupled potential integral equations.
Using basic pseudo-differential techniques, discrete stability is obtained for the standard $L^2$-pairing, 
and we deduce spectral convergence for high-order schemes employing discontinuous basis functions, 
or with additional element-wise tangential continuity imposed in the vectorial part of the system. 
The schemes are shown to be much better conditioned than a mature commercial EFIE/CFIE solver. 
In many cases appearing almost independent of wavenumber, mesh density, and the order.
\end{abstract}



\section{Introduction} 
The standard EFIE/MFIE/CFIE are plagued by drawbacks.
Namely, the resonant frequencies in EFIE/MFIE, dense mesh breakdown in EFIE, 
and potentially catastrophic cancellations in MFIE as the wavenumber $k$ goes to zero.
Moreover, if the genus is non-zero, all suffer from topological breakdown in this limit.
At $k=0$ the kernel of EFIE consists of all the harmonic vector fields on the scatterer, 
and MFIE, in turn, has kernel consisting of all the harmonic Neumann vector fields.
But this results in ill-conditioned systems on the operator level, before discretization, 
and no trick post-discretization can overcome this flaw in the equations themselves.\\

The decoupled potential integral equations (DPIE) \cite{VicoGreengard2016, VicoGreengard2025} avoid these problems.
But their general functional analytic theory (as in \cite{Buffa2003, Hsiao1997}) has not yet been developed.
At the time of writing, Galerkin boundary element method implementations do exist, 
but only at low order \cite{Baumann2023, Baumann2022, LiShanker2019}, using roof-top Rao-Wilton-Glisson basis functions, 
whereas a lot of effort seems to have focused on high-order Nystr\"om solvers \cite{VicoFerrando2015, VicoGreengard2025}.
The DPIE system is essentially just a copy of the modified integral equations in \cite{ColtonKress2013}, but augmented with a zero-mean constraint to kill the kernel in the static limit $k=0$.
This is analogous to adding a gauge-invariant consistency condition as is done in \cite{EpsteinGreengard2013},
and is known well-posed in the framework of H\"older spaces \cite{VicoGreengard2016, VicoGreengard2025} at operator level.
Related to the DPIE system are the decoupled field integral equations (DFIE) of \cite{VicoGreengard2018}, 
which solve the transmission problem (as in \cite{Baumann2022, LiShanker2019}, but the equations are different). 
Again the authors work in the H\"older spaces, and make no mention of Sobolev spaces.
Nevertheless, these advances suggest potential for broad application. \\

Our objective is to prove well-posedness in the Sobolev scale for smooth surfaces.
The physics will be constrained to closed perfect electrical conductors, as in \cite{VicoGreengard2016, VicoGreengard2025},
and our main results are Theorems~\ref{theorem:DPIE-well-posed} and \ref{theorem:DPIE-galerkin-stable}. 
We then discretize in the $L^2$-pairing,
but construct a scheme using \textit{higher-order} tensor product Legendre basis functions.
Upon integrating the $L^2$-pairing by parts, we are left with only weakly singular kernels, 
but pay with residual line integrals, which must either be canceled out or be computed.
It turns out to be advantageous to enforce tangential continuity in the vector basis,
which is precisely the right constraint that cancels all the line integral contributions.
On the other hand, a discontinuous higher-order basis can be made fully orthogonal,
and, as we shall observe, results in both mesh and order independent conditioning.
Nevertheless, both approaches produce well-conditioned system matrices.



\newpage
In this article, $k\geq 0$ is the wavenumber of an unbounded homogeneous medium.
The scatterer subset is open, bounded, and has a smooth oriented boundary $\Gamma\subset \mathbb{R}^3$.
It is allowed to have $N\geq 1$ components $\Gamma = \sqcup_{j=1}^N \Gamma_j $, disjoint and multiply connected.
Associated to these surfaces is $\nu : \Gamma \to \mathbb{R}^3$, the unit outward normal vector field to $\Gamma$.
The Helmholtz Green's function for $k$ is
\begin{align*}
    \Phi_k : \{ (x,y) \in \mathbb{R}^3 \times \mathbb{R}^3 \, | \, x \neq y\} \to \mathbb{C} : (x,y) \mapsto \frac{1}{4\pi}\frac{e^{ik|x-y|}}{|x-y|}.
\end{align*}
and we write $g$ for the metric on $\Gamma$, pullback of the Euclidean metric by the inclusion.
Given $\rho \in C^\infty(\Gamma)$, we define at any point $x\in \Gamma$ the integrals
\begin{align*}
    S_k \rho(x) &= \int_{\Gamma} \Phi_k(x,y) \rho(y) \, d\sigma(y), \\
    D_k \rho(x) &= \int_{\Gamma} \Big[\frac{\partial}{\partial \nu(y)}\Phi_k(x,y) \Big]\rho(y) \, d\sigma(y),
\end{align*}
and 
\begin{align*}
    D_k^t \rho(x) &= \int_{\Gamma} \Big[\frac{\partial}{\partial \nu(x)}\Phi_k(x,y) \Big]\rho(y) \, d\sigma(y), \\
    W_k \rho(x) &= \lim_{\varepsilon \to 0}\int_{\Gamma\setminus \{ z\in \Gamma \, | \, |z-x|< \varepsilon \}} \Big[ \frac{\partial}{\partial \nu(x)}\frac{\partial}{\partial \nu(y)}\Phi_k(x,y) \Big] \rho(y) \, d \sigma(y),
\end{align*}
where $d\sigma$ is the surface measure, the contraction of Euclidean volume with $\nu$ on $\Gamma$.
These are the single and double layer potentials, as well as the transpose operator $D_k^t$,
and the hyper singular operator $W_k$, written explicitly as a Cauchy principal value.
Given $a\in C^\infty(\Gamma, T\Gamma \otimes_{\mathbb{R}} \mathbb{C} )$, we also define
\begin{align*}
    M_ka(x) = \int_\Gamma \nu(x) \times (\nabla_x \Phi_k(x,y) \times a(y)) \, d \sigma(y),
\end{align*}
whose complete integrand is only weakly singular, hence requires no interpretation.
This is the magnetic operator, which appears in the leading part of the vector DPIE.
We write $A\in \Psi^d(\Gamma; E, F)$ for an order $d\in \mathbb{R}$ classical pseudo-differential operator
\begin{align*}
    A  : C^\infty(\Gamma; E) \to C^\infty(\Gamma; F)
    \quad
    \textnormal{with principal symbol}
    \quad 
    \sigma^d(A),
\end{align*}
which acts on smooth sections of smooth complex vector bundles $E$ and $F$ over $\Gamma$.
An account of their properties is found in \cite{HIII, TaylorPsiDO}, or \cite{AlinhacGerard2007} for a readable introduction.
Moreover, we put $\Psi^d(\Gamma;E) = \Psi^d(\Gamma;E,E)$ and $T_\mathbb{C} \Gamma = T\Gamma \otimes_{\mathbb{R}} \mathbb{C}$ for short.

\begin{theorem}\label{theorem:layer-psido}
The above operators are all classical pseudo-differential operators.
Importantly, their orders are as follows:
\medskip
\begin{enumerate}
    \item $S_k, D_k, D_k^t \in \Psi^{-1}(\Gamma)$.
    \medskip
    \item $W_k \in \Psi^1(\Gamma)$ but $W_k - W_0 \in \Psi^{-1}(\Gamma)$.
    \medskip
    \item $M_k \in \Psi^{-1}(\Gamma, T_\mathbb{C}\Gamma)$.
\end{enumerate}
\medskip
Moreover, the principal symbol of $S_k$ is given by
\begin{align*}
    \sigma^{-1}(S_k)(x,\xi) = \frac{1}{2|\xi|_g} 
    \quad
    \textnormal{for any}
    \quad
    (x,\xi) \in T^* \Gamma\setminus 0.
\end{align*}
\end{theorem}

\newpage
\section{Invertibility and Stability} 
The DPIE is well-understood in H\"older spaces.
Our aim here is to extend this understanding to the standard Sobolev space scale, 
and show that the resulting pseudo-differential system is actually Fredholm on these.
It is even strongly elliptic, which immediately implies a suitable Gårding inequality. A basic Schatz-type argument then suffices to get discrete stability in the $L^2$-pairing. \\

Let $\chi_{\Gamma_j} : \Gamma \to \{0,1\}$ be the characteristic functions for the $N$ disjoint pieces of $\Gamma$.
Initially, we restrict to $a,f\in C^{\infty}(\Gamma, T_\mathbb{C}\Gamma)$, $\rho, h\in C^{\infty}(\Gamma)$, and $(v_j)_{j=1}^N, (q_j)_{j=1}^N\in \mathbb{C}^N$.
The vector DPIE of \cite{VicoGreengard2016, VicoGreengard2025} is the system of integral equations
\begin{align*}
\Big(
\begin{bmatrix}
    \frac{1}{2}I + M_k & L_k \\
    0 &  \frac{1}{2}I + D_k
\end{bmatrix}
+
i\eta
\begin{bmatrix}
   JS_kJ & \textnormal{curl}_\Gamma S_k  \\
   S_k \textnormal{curl}_\Gamma & -k^2 S_k
\end{bmatrix}
\Big)
\begin{bmatrix}
    a \\
    \rho
\end{bmatrix}
-
\begin{bmatrix}
    0 \\
    \sum_{j=1}^N v_j \chi_{\Gamma_j}
\end{bmatrix}
&=
\begin{bmatrix}
    f \\
    h
\end{bmatrix}, \\
\int_{\Gamma_j} - \nu \cdot S_k (\nu\rho) + i\eta\Big(\nu \cdot S_k Ja -\frac{1}{2}\rho + D^t_k \rho \Big) \, d\sigma &= q_j,
\end{align*}
and the scalar DPIE is the simpler system
\begin{align*}
\Big(
\frac{1}{2}I + D_k
-
i\eta
S_k 
\Big)
\rho
- \sum_{j=1}^N v_j \chi_{\Gamma_j}
&=
h, \\
\int_{\Gamma_j} (W_k - W_0)\rho + i\eta\Big(\frac{1}{2} - D^t_k \Big)\rho \, d\sigma &=  q_j ,
\end{align*}
where $J : C^\infty(\Gamma, T_\mathbb{C}\Gamma) \to C^\infty(\Gamma, T_\mathbb{C}\Gamma) : a \mapsto \nu \times a$ is the almost complex structure.
Notation is overloaded so that the domain of $J$ is replaced by $C^\infty(\Gamma, \mathbb{C}^3)$ when needed.
The value $\eta\in \mathbb{C}$ is free, but controls whether the system is uniquely solvable or not.
Finally, the operator $L_k \in \Psi^{-1}(\Gamma, \mathbb{C}, T_\mathbb{C}\Gamma)$ is simply  
\begin{align*}
    L_k : C^\infty(\Gamma) \to C^\infty(\Gamma, T_\mathbb{C}\Gamma) : \rho \mapsto -JS_k(\nu\rho).
\end{align*}
Once extended to H\"older spaces, the above two systems are well-posed if $\eta \in \mathbb{R}\setminus \{0\}$. 
This is one of the main results of \cite{VicoGreengard2016}.
In our formulation, we have used the identities
\begin{align*}
    \nabla \cdot S_k(\nu \times a)|_\Gamma &= S_k (\textnormal{div}_\Gamma (\nu\times a)) = S_k \textnormal{curl}_{\Gamma}(a), \\
    \nu \times \nabla_\Gamma S_k \rho &= \textnormal{curl}_\Gamma S_k \rho,
\end{align*}
where the first is the jump relation \cite[Theorem 2.17]{ColtonKress2013} followed by integration parts.
This only holds because $a$ is assumed smooth, it fails for merely H\"older continuous $a$.
But the right side extends by duality to distributions on $\Gamma$, hence to Sobolev spaces. 
The vector curl is defined $\textnormal{curl}_{\Gamma}(a) = \textnormal{div}_\Gamma(\nu \times a)$, and $\textnormal{div}_\Gamma$ is the surface divergence.
In the following, we write
\begin{align*}
\mathcal{A}
=
\begin{bmatrix}
    \frac{1}{2}I + M_k & L_k \\
    0 &  \frac{1}{2}I + D_k
\end{bmatrix}
+
i\eta
\begin{bmatrix}
   JS_kJ & \textnormal{curl}_\Gamma S_k  \\
   S_k \textnormal{curl}_\Gamma & -k^2 S_k
\end{bmatrix}
:
\begin{matrix}
C^{\infty}(\Gamma, T_\mathbb{C}\Gamma) \\
\oplus \\
C^{\infty}(\Gamma)
\end{matrix}
\to
\begin{matrix}
C^{\infty}(\Gamma, T_\mathbb{C}\Gamma) \\
\oplus \\
C^{\infty}(\Gamma)
\end{matrix}
\end{align*}
and
\begin{align*}
\mathcal{B} = \frac{1}{2}I + D_k
-
i\eta
S_k  
:
C^{\infty}(\Gamma)
\to 
C^{\infty}(\Gamma),
\end{align*}
which are evidently both pseudo-differential of order $0$.

\newpage
\begin{lemma}
$\mathcal{A} \in\Psi^0(\Gamma; T_\mathbb{C} \Gamma \oplus \mathbb{C})$ is an elliptic operator for all $\eta \in \mathbb{C}\setminus \{ -i, +i\}$.
A realization of $\mathcal{A}$ exists for every $s\in \mathbb{R}$ as a Fredholm operator
\begin{align*}
    \mathcal{A}
    : \,
\begin{matrix}
H^{s}(\Gamma, T_\mathbb{C}\Gamma) \\
\oplus \\
H^s(\Gamma)
\end{matrix}
\to
\begin{matrix}
H^{s}(\Gamma, T_\mathbb{C}\Gamma) \\
\oplus \\
H^s(\Gamma)
\end{matrix}
:
\begin{bmatrix}
    a \\
    \rho
\end{bmatrix}
\mapsto
\Big(
\begin{bmatrix}
    \frac{1}{2}I & i\eta \,\textnormal{curl}_\Gamma S_k \\
    i\eta \, S_k \textnormal{curl}_\Gamma &  \frac{1}{2}I
\end{bmatrix}
+
\mathcal{K}_\eta
\Big)
\begin{bmatrix}
    a \\
    \rho
\end{bmatrix},
\end{align*}
where $\mathcal{K}_\eta \in \Psi^{-1}(\Gamma; T_\mathbb{C}\Gamma \oplus \mathbb{C})$, and $\textnormal{ind}\, \mathcal{A} = 0$ for these $\eta$, but independently of $s\in \mathbb{R}$.
The principal symbol is strongly elliptic for $\eta \in \mathbb{R}$, and is of the form
\begin{align*}
    \sigma^0(\mathcal{A})(x,\xi)
    \begin{bmatrix}
     \delta a \\
     \delta \rho
    \end{bmatrix}
    =
    \begin{bmatrix}
        \frac{1}{2}\delta a - \frac{1}{2}\eta \Big( \nu(x) \times \frac{\xi^\sharp}{|\xi|_g}\Big)\delta \rho \\
        +\frac{1}{2}\eta\Big(\nu(x) \times \frac{\xi^\sharp}{|\xi|_g}\Big)\cdot \delta a + \frac{1}{2} \delta \rho
    \end{bmatrix},
\end{align*}
where $(\delta a, \delta \rho) \in (T_x \Gamma \otimes_\mathbb{R} \mathbb{C}) \oplus \mathbb{C}$ is a complex vector over $(x,\xi) \in T^*_x\Gamma \setminus \{0\}$ for $x\in \Gamma$.
\end{lemma}
\begin{proof}
Invoking Theorem~\ref{theorem:layer-psido}, $\mathcal{A}$ is indeed pseudo-differential of the above form. 
First, the two order $0$ off-diagonal terms are
\begin{align*}
    \textnormal{curl}_\Gamma S_k \in \Psi^0(\Gamma; \mathbb{C}, T_\mathbb{C}\Gamma)
    \quad
    \textnormal{and}
    \quad
    S_k \textnormal{curl}_\Gamma \in \Psi^0(\Gamma; T_\mathbb{C}\Gamma, \mathbb{C}),
\end{align*}
where we abuse notation, as the curl operators are different, acting on different spaces.
Using the symbolic calculus, if $(x,\xi)\in T_x^*\Gamma \setminus \{0\}$ we get
\begin{align*}
    \sigma^0(\textnormal{curl}_\Gamma S_k )(x,\xi) \delta \rho 
    &= \sigma^1(\textnormal{curl}_\Gamma)(x,\xi)\sigma^{-1}(S_k)(x,\xi) \delta \rho \\
    &=
    \frac{1}{2}i\Big(\nu(x) \times \frac{\xi^\sharp}{|\xi|_g}\Big) \delta \rho,
\end{align*}
and
\begin{align*}
    \sigma^0( S_k \textnormal{curl}_\Gamma )(x,\xi) \delta a 
    &= \sigma^{-1}(S_k)(x,\xi)  \sigma^1( \textnormal{curl}_\Gamma )(x,\xi) \delta a \\
    &=
    -\frac{1}{2}i\Big(\nu(x) \times \frac{\xi^\sharp}{|\xi|_g} \Big) \cdot \delta a.
\end{align*}
The principal symbol is of the stated form, and it remains to see that it is invertible. 
Note that we have the orthogonal decomposition
\begin{align*}
    T_x\Gamma \otimes_\mathbb{R} \mathbb{C} = \textnormal{span}_\mathbb{C}\Big\{\frac{\xi^\sharp}{|\xi|_g}\Big\} \oplus \textnormal{span}_\mathbb{C}\Big\{\nu(x) \times \frac{\xi^\sharp}{|\xi|_g}\Big\},
\end{align*}
and that, in this basis, $\sigma^0(\mathcal{A})(x,\xi)$ is just
\begin{align*}
\frac{1}{2}
    \begin{bmatrix}
        1 & 0 & 0 \\
        0 & 1 & -\eta \\
        0 & +\eta & 1 
    \end{bmatrix},
\end{align*}
which is clearly invertible for all $\eta \in \mathbb{C}\setminus \{-i, i\}$ and also strongly elliptic for all $\eta \in \mathbb{R}$.
Now provided that $|\eta|$ is sufficiently small, we have
\begin{align*}
\textnormal{ind}\, \mathcal{A}
&=
\textnormal{ind}\,
\begin{bmatrix}
    \frac{1}{2}I & i\eta \,\textnormal{curl}_\Gamma S_k \\
    i\eta \, S_k \textnormal{curl}_\Gamma &  \frac{1}{2}I
\end{bmatrix}\Big(I + \begin{bmatrix}
    \frac{1}{2}I & i\eta \,\textnormal{curl}_\Gamma S_k \\
    i\eta \, S_k \textnormal{curl}_\Gamma &  \frac{1}{2}I
\end{bmatrix}^{-1}\mathcal{K}_\eta\Big)
= 0.
\end{align*}
But index homotopy invariance then implies $\textnormal{ind}\, \mathcal{A} = 0$ for all $\eta \in \mathbb{C}\setminus \{-i, +i\}$.
\end{proof}

\newpage
An extension of the vector DPIE to the full Sobolev space scale is now possible.
The surface zero-mean constraints extend to bounded functionals
\begin{align*}
(\zeta_\mathcal{A})_j : H^s(\Gamma, T_\mathbb{C}\Gamma \oplus \mathbb{C}) \to \mathbb{C} 
:  (a, \rho) \mapsto \Big\langle \chi_{\Gamma_j}, i \eta ( \nu\cdot S_kJa - \frac{1}{2}\rho +D_k^t\rho) - \nu\cdot S_k (\nu \rho) \Big\rangle,
\end{align*}
where the brackets $H^{-s}(\Gamma) \times H^s(\Gamma) : (\phi, \psi) \mapsto \langle \phi, \psi \rangle$ indicate the $L^2$-duality pairing.
This allows us to write the full vector DPIE operator
\begin{align*}
    \mathcal{V}
        : \,
\begin{matrix}
H^{s}(\Gamma, T_\mathbb{C}\Gamma \oplus \mathbb{C}) \\
\oplus \\
\mathbb{C}^N
\end{matrix}
\to
\begin{matrix}
H^{s}(\Gamma, T_\mathbb{C}\Gamma \oplus \mathbb{C}) \\
\oplus \\
\mathbb{C}^N
\end{matrix}
:
\begin{bmatrix}
    (a, \rho)  \\
    v 
\end{bmatrix}
\mapsto
\begin{bmatrix}
    \mathcal{A} & \tau_\mathcal{A}  \\
    \zeta_\mathcal{A} & 0
\end{bmatrix}
\begin{bmatrix}
    (a,\rho) \\
    v 
\end{bmatrix},
\end{align*}
where
\begin{align*}
    \tau_\mathcal{A} : \mathbb{C}^N \to C^\infty(\Gamma, T_\mathbb{C}\Gamma \oplus \mathbb{C}) : v \mapsto 
    \begin{bmatrix}
        0 \\
        -\sum_{j=1}^N v_j \chi_{\Gamma_j}
    \end{bmatrix},
\end{align*}
and $\zeta_\mathcal{A} = \oplus_{j=1}^N (\zeta_\mathcal{A})_j $ is the combined functional into $\mathbb{C}^N$.

\begin{proposition}\label{proposition:vector-well-posed}
$\mathcal{V}$ is invertible for every $s\in \mathbb{R}$ whenever $\eta \in \mathbb{R}\setminus\{0\}$.   
\end{proposition}
\begin{proof}
Observe first that $\mathcal{V}$ is again a Fredholm operator, again with $\textnormal{ind}\, \mathcal{V} = 0$.
This is simply because 
\begin{align*}
    \mathcal{V}
    =
\begin{bmatrix}
    \mathcal{A} & 0  \\
    0 & I
\end{bmatrix}
+
\begin{bmatrix}
    0 & \tau_\mathcal{A}  \\
    \zeta_\mathcal{A} & -I
\end{bmatrix},
\end{align*}
and the latter operator is a compact perturbation, so that $\textnormal{ind} \, \mathcal{V} = \textnormal{ind} \, \mathcal{A} = 0 $ holds.
Now if $\mathcal{V} ((a,\rho), v) = 0$, we have
\begin{align*}
    \mathcal{A} (a,\rho) = \Big(0, \sum_{j=1}^N v_j \chi_{\Gamma_j} \Big) \in C^\infty(\Gamma, T_\mathbb{C}\Gamma \oplus \mathbb{C}).
\end{align*}
But $\mathcal{A}$ is elliptic, so $(a, \rho) \in C^\infty(\Gamma, T_\mathbb{C}\Gamma \oplus \mathbb{C})$, hence is a smooth homogeneous solution,
which implies that $((a,\rho), v) = 0$, because we have uniqueness for smooth functions. 
The argument is found in the proof of \cite[Theorem 3.14]{VicoGreengard2016} using \cite[Theorem 4.42]{ColtonKress2013}.
Consequently, $\mathcal{V}$ is injective of index zero, so it is invertible.
\end{proof}

Obviously, $\mathcal{B} \in \Psi^0(\Gamma)$ is strongly elliptic for all $\eta\in \mathbb{C}$ with symbol $\sigma^0(\mathcal{B})=\frac{1}{2}$.
The associated zero-mean constraints extend similarly to bounded functionals
\begin{align*}
    (\zeta_\mathcal{B})_j : H^s(\Gamma) \to \mathbb{C} : \rho \mapsto \Big\langle \chi_{\Gamma_j}, i\eta\Big(\frac{1}{2} - D^t_k \Big)\rho + (W_k - W_0)\rho  \Big\rangle,
\end{align*}
and the full scalar DPIE operator is Fredholm
\begin{align*}
    \mathcal{S}
        : \,
\begin{matrix}
H^{s}(\Gamma) \\
\oplus \\
\mathbb{C}^N
\end{matrix}
\to
\begin{matrix}
H^{s}(\Gamma) \\
\oplus \\
\mathbb{C}^N
\end{matrix}
:
\begin{bmatrix}
    \rho \\
    v
\end{bmatrix}
\mapsto
\begin{bmatrix}
    \mathcal{B} & \tau_\mathcal{B} \\
    \zeta_\mathcal{B} & 0 
\end{bmatrix}
\begin{bmatrix}
    \rho \\
    v
\end{bmatrix}
\end{align*}
where
\begin{align*}
    \tau_\mathcal{B} : \mathbb{C}^N \to C^\infty(\Gamma) : v \mapsto -\sum_{j=1}^N v_j \chi_{\Gamma_j}
\end{align*}
and again $\zeta_\mathcal{B} = \oplus_{j=1}^N (\zeta_\mathcal{B})_j $.

\newpage
\begin{proposition} \label{proposition:scalar-well-posed}
    $\mathcal{S}$ is invertible for every $s\in \mathbb{R}$ whenever $\eta \in \mathbb{R}\setminus\{0\}$.
\end{proposition}
\begin{proof}
The argument is basically the same as for the vector operator, but simpler.
Clearly, $\textnormal{ind}\, \mathcal{S} = 0$ for all $\eta \in \mathbb{C}$, because $\mathcal{B}$, hence $\mathcal{S}$, is a compact perturbation of $\frac{1}{2}I$.
Just as before, if $\mathcal{S}(\rho, v) = 0$, then 
\begin{align*}
    \mathcal{B}\rho = \sum_{j=1}^N v_j \chi_{\Gamma_j} \in C^\infty(\Gamma),
\end{align*}
and $\mathcal{B}$ is elliptic, so $\rho \in C^\infty(\Gamma)$, hence we are back in the realm of smooth functions.
The proof of \cite[Theorem 3.13]{VicoGreengard2016} using \cite[Theorem 3.33]{ColtonKress2013} shows that $(\rho, v) = 0$. 
\end{proof}

Combining Propositions~\ref{proposition:scalar-well-posed} and \ref{proposition:vector-well-posed}, we conclude:

\begin{theorem}\label{theorem:DPIE-well-posed}
    The DPIE system is well-posed for all $s\in \mathbb{R}$ and $\eta \in \mathbb{R}\setminus\{0\}$.
\end{theorem}

Operators $\mathcal{A}$ and $\mathcal{B}$, being strongly elliptic, admit strong Gårding inequalities.
This statement is a foundational result in the calculus of pseudo-differential operators, and the theorem can be found in \cite{HIII} or \cite{TaylorPsiDO} with more generality than needed here.
In our case, it means that there are constants

\bigskip
\begin{enumerate}
    \item $\alpha_0, \alpha_{-\frac{1}{2}}>0$ such that if $f=(a,\rho) \in L^2(\Gamma, T_\mathbb{C}\Gamma \oplus \mathbb{C})$ then
\begin{align*}
    \textnormal{Re} \, \langle\mathcal{A} f, f\rangle_{L^2(\Gamma, T_\mathbb{C}\Gamma \oplus \mathbb{C})} \geq \alpha_0||f||^2_{L^2(\Gamma, T_\mathbb{C}\Gamma \oplus \mathbb{C})} - \alpha_{-\frac{1}{2}}||f||^2_{H^{-\frac{1}{2}}(\Gamma, T_\mathbb{C}\Gamma \oplus \mathbb{C})},
\end{align*}
    \item $\beta_0, \beta_{-\frac{1}{2}}>0$ such that if $\rho \in L^2(\Gamma)$ then
\begin{align*}
      \textnormal{Re} \, \langle \mathcal{B} \rho, \rho\rangle_{L^2(\Gamma)} \geq \beta_0||\rho||^2_{L^2(\Gamma)} - \beta_{-\frac{1}{2}}||\rho||^2_{H^{-\frac{1}{2}}(\Gamma)},  
\end{align*}
\end{enumerate}

\noindent
which depend on the geometry of $\Gamma$, and $\alpha_{-\frac{1}{2}},\beta_{-\frac{1}{2}}$ also depend on the wavenumber $k$. 
The objective is now to obtain discrete stability in the case $s=0$ in the $L^2$-pairing, and the above two Gårding inequalities 
are going to be instrumental in achieving this.
Henceforth, we make the dependence on $k$ explicit; $\mathcal{A}_k$, $\mathcal{B}_k$, $(\zeta_{\mathcal{A}})_k$, $(\zeta_{\mathcal{B}})_k$, $\mathcal{V}_k$, and $\mathcal{S}_k$. Obviously, the maps $\tau_{\mathcal A}$ and $\tau_{\mathcal{B}}$ do not depend on $k$.\\

Let $\{H_n\}_{n=1}^\infty$ be a sequence of finite-dimensional subspaces of a Hilbert space $H$.
It is said to be a Galerkin sequence if
\begin{align*}
    \inf_{w\in H_n} || u - w || \to 0
    \quad
    \textnormal{as}
    \quad
    n \to \infty
    \quad
    \textnormal{for any}
    \quad
    u\in H,
\end{align*}
and we say that the DPIE is $L^2$-Galerkin stable if for any $K>0$ the following holds:

\bigskip
\begin{enumerate}
\item If $H = L^2(\Gamma, T_\mathbb{C}\Gamma \oplus \mathbb{C}) \oplus \mathbb{C}^N$, there is some $n_0$ such that if $n\geq n_0$ then
\begin{align*}
    \inf_{u\in H_n\setminus \{0\}} \sup_{w\in H_n\setminus \{0\}} \frac{|\langle \mathcal{V}_k u, w \rangle_{H}|}{||u||_{H} ||w||_{H}} \geq \alpha_K > 0
    \quad
    \textnormal{for all}
    \quad
    k\in[0,K].
\end{align*}

\item If $H = L^2(\Gamma) \oplus \mathbb{C}^N$, there is some $n_0$ such that if $n\geq n_0$ then
\begin{align*}
    \inf_{u\in H_n\setminus \{0\}} \sup_{w\in H_n\setminus \{0\}} \frac{|\langle \mathcal{S}_k u, w \rangle_{H}|}{||u||_{H} ||w||_{H}} \geq \beta_K > 0
    \quad
    \textnormal{for all}
    \quad
    k\in[0,K].
\end{align*}
\end{enumerate}

\newpage
The inf-sup conditions ensure that the discrete operators are uniformly invertible.
A discrete version of Babu\v{s}ka’s theorem implies the quasi-optimal estimate 
\begin{align*}
    || \mathcal{V}^{-1}_k g - u_n ||_{H} \leq \frac{||\mathcal{V}_k||}{\alpha_K} \inf_{w \in H_n} || \mathcal{V}^{-1}_kg  - w ||_{H},
\end{align*}
where $g \in H$ and $u_n\in H_n$ is the unique solution to
\begin{align*}
    \langle \mathcal{V}_k u_n, w \rangle_{H} = \langle g, w \rangle_{H}
    \quad
    \textnormal{for all}
    \quad
    w \in H_n.
\end{align*}
But the subspaces approximate $\mathcal{V}^{-1}_kg$ arbitrarily well, so convergence is guaranteed.
Naturally, the same quasi-optimality holds for $\mathcal{S}_k$ and its discrete solutions. \\

Write $H^s = H^s(\Gamma, T_\mathbb{C}\Gamma \oplus \mathbb{C})$. We prove the discrete inf-sup condition only for $\mathcal{V}_k$.
Although not exactly the same, the proof for $\mathcal{S}_k$ is very similar.

\begin{lemma} \label{lemma:Ak-continuity}
$[0,K] \to B(H^0) : k \mapsto \mathcal{A}_k$ is Lipschitz continuous.
\end{lemma}
\begin{proof}
Observe that for $f\in H^0$ and $x\in \Gamma$, the difference at $k_0\in [0,K]$ is
\begin{align*}
    (\mathcal{A}_k - \mathcal{A}_{k_0}) f(x) = \int_{\Gamma} (P_k\Phi_k - P_{k_0} \Phi_{k_0})(x,y) f(y) \, d\sigma(y),
\end{align*}
where $P_k$ is a matrix of differential operators in $(x,y)$ with only smooth coefficients.
It is of order at most $1$, and of the form $P_k = P - i\eta k^2 E_{2,2}$
with $P$ independent of $k$, and $E_{2,2}$ is the matrix with $1$ in entry $(2,2)$, corresponding to scalars, and zero else.
The kernel is estimated by
\begin{align*}
    ||(P_k\Phi_{k} - P_{k_0}\Phi_{k_0})(x,y)||
    &=
   ||P_k(\Phi_{k} - \Phi_{k_0})(x,y) + (P_k-P_{k_0})\Phi_{k_0}(x,y)|| \\
    &=
    \Big|\Big|\frac{i}{4\pi}
    \int^{k}_{k_0} P_k[e^{is|x-y|} ]\, ds -i\eta (k^2 -k_0^2)\Phi_{k_0}(x,y) E_{2,2} \Big|\Big|\\
    &\leq
    \frac{1}{4\pi}\Big|\int_{k_0}^k ||P_k[e^{is|x-y|}]||\, ds\Big| + |\eta|\frac{|k^2-k_0^2|}{4\pi|x-y|} \\
    &\leq 
    C_{K,\eta} |k-k_0|\Big(1+ \frac{1}{|x-y|}\Big),
\end{align*}
where $C_{K,\eta} >0$ depends on $K$, $\eta$, and the smooth coefficients of $P$, but not $k$.
Applying the Schur test, we obtain
\begin{align*}
    || \mathcal{A}_k - \mathcal{A}_{k_0} ||_{B(H^0)}
    \leq C_{K,\eta} |k-k_0|
    \sup_{x\in \Gamma} \int_\Gamma \Big(1+ \frac{1}{|x-y|}\Big) \, d\sigma(y),
\end{align*}
which completes the proof.
\end{proof}

\begin{corollary} \label{corollary:zetak-continuity}
$[0,K] \to B(H^{0},\mathbb{C}^N) : k\mapsto (\zeta_\mathcal{A})_k$ is Lipschitz continuous.
\end{corollary}
\begin{proof}
Likewise, for $f\in H^0$, the difference at $k_0\in [0,K]$ is
\begin{align*}
    ((\zeta_{\mathcal{A}})_k - (\zeta_{\mathcal{A}})_{k_0}) f = \Big(\int_{\Gamma_j} \Big[ \int_\Gamma P(\Phi_k-\Phi_{k_0})(x,y) f(y) d\sigma(y) \Big] \,   d\sigma(x) \Big)_{j=1}^N,
\end{align*}
where $P$ is a $1 \times 2$-formed differential operator in $(x,y)$ with only smooth coefficients.
It is of order at most $1$, and it does not depend on $k$, unlike in the above the argument. Repeating that and using the Cauchy-Schwarz inequality completes the proof.
\end{proof}

\newpage
\begin{theorem} \label{theorem:DPIE-galerkin-stable}
The DPIE is $L^2$-Galerkin stable if $\eta \in \mathbb{R}\setminus\{0\}$.    
\end{theorem}
\begin{proof}
Arguing by contradiction, suppose $\mathcal{V}_k$ is not uniformly stable in $k\in [0,K]$.
In that case, find $u_j \in H_{n_j}$ and $k_j\in [0,K]$, with $ n_j$ strictly increasing, such that
\begin{align*}
||u_j||_{H^0 \oplus \mathbb{C}^N} = 1
\quad
\textnormal{and}
\quad
\sup_{w\in H_{n_j}} \frac{|\langle\mathcal{V}_{k_j} u_j, w\rangle_{H^0 \oplus \mathbb{C}^N}|}{||w||_{H^0 \oplus \mathbb{C}^N}} < \frac{1}{j},
\end{align*}
where by compactness, passing to a subsequence, we may assume $k_j \to k_*\in [0,K]$.  
Now if $w\in H^0 \oplus \mathbb{C}^N$, we can take $H_{n_j} \ni w_j \to w$ and estimate
\begin{align*}
|\langle\mathcal{V}_{k_j} u_{j}, w \rangle_{H^0 \oplus \mathbb{C}^N}| 
\leq
\frac{1}{j}||w_j||_{H^0 \oplus \mathbb{C}^N} + \sup_{k\in [0,K]}
    ||\mathcal{V}_k||_{B(H^{0}\oplus \mathbb{C}^N)} ||w-w_j||_{H^0 \oplus \mathbb{C}^N}.
\end{align*}
This means that $\mathcal{V}_{k_j} u_j$ converges weakly to zero in $H^0\oplus \mathbb{C}^N$ thus also in $H^{-\frac{1}{2}}\oplus \mathbb{C}^N$.
Note that the supremum over $k$ above is finite by Lemma~\ref{lemma:Ak-continuity} and its Corollary~\ref{corollary:zetak-continuity}.
In particular, we can also take $w = u_j$ in the above to get
\begin{align*}
\delta_j = \textnormal{Re} \, \langle\mathcal{V}_{k_j} u_j, u_j\rangle_{H^0 \oplus \mathbb{C}^N} < \frac{1}{j}.
\end{align*}
But viewing $(u_j)_{j=1}^\infty$ within $H^{-\frac{1}{2}} \oplus \mathbb{C}^N$, it has a strongly convergent subsequence. 
Consequently, without loss of generality, we may assume that
\begin{align*}
    u_j \to u=(f,v)
    \quad
    \textnormal{in}
    \quad
    H^{-\frac{1}{2}} \oplus \mathbb{C}^N
    \quad
    \textnormal{as}
    \quad
    j \to \infty, 
\end{align*}
and writing $u_j = (f_j, v_j)$, we compute
\begin{align*}
    \delta_j
    &= 
    \textnormal{Re} \langle\mathcal{A}_{k_j} f_j + \tau_\mathcal{A}(v_j), f_j \rangle_{H^0}
    +
    \textnormal{Re} \, \langle (\zeta_\mathcal{A})_{k_j}(f_j), v_j \rangle_{\mathbb{C}^N} \\
    &\geq \textnormal{Re} \langle\mathcal{A}_{k_j} f_j, f_j \rangle_{H^0} - |\langle (\zeta_\mathcal{A})_{k_j}(f_j), v_j \rangle_{\mathbb{C}^N}| - C ||f_j||_{H^{-\frac{1}{2}}}|v_j| \\
    &\geq 
    \textnormal{Re} \langle\mathcal{A}_{k_j} f_j, f_j \rangle_{H^0} 
    - (|(\zeta_\mathcal{A})_{k_*}(f_j)| + ||(\zeta_\mathcal{A})_{k_*} - (\zeta_\mathcal{A})_{k_j}||_{B(H^0,\mathbb{C}^N)}  
    + C ||f_j||_{H^{-\frac{1}{2}}})|v_j|.
\end{align*}
The constant $C>0$ here is chosen so that $||\tau_\mathcal{A}(v')||_{H^{\frac{1}{2}}} \leq C|v'|$ holds for all $v'\in \mathbb{C}^N$.
Using the Gårding inequality at $k_*$, this gives
\begin{align*}
    \delta_j + \alpha_{-\frac{1}{2}}||f_j||^2_{H^{-\frac{1}{2}}} + \alpha_0|v_j|^2 + ||\mathcal{A}_{k_j} -\mathcal{A}_{k_*}||_{B(H^0)} \geq \alpha_0 - 
    \delta_j' |v_j|,
\end{align*}
where we have put
\begin{align*}
   \delta_j' = |(\zeta_\mathcal{A})_{k_*}(f_j)| + ||(\zeta_\mathcal{A})_{k_*} - (\zeta_\mathcal{A})_{k_j}||_{B(H^0,\mathbb{C}^N)} + C ||f_j||_{H^{-\frac{1}{2}}}. 
\end{align*}
Finally, letting $j\to \infty$ on both sides, by Lemma~\ref{lemma:Ak-continuity} and its Corollary~\ref{corollary:zetak-continuity}, we get
\begin{align*}
   \alpha_{-\frac{1}{2}}||f||^2_{H^{-\frac{1}{2}}} + \alpha_0|v|^2 + (|(\zeta_\mathcal{A})_{k_*} (f)| + C ||f||_{H^{-\frac{1}{2}}} )|v| \geq \alpha_0> 0,
\end{align*}
which shows that $u =(f,v) \neq 0$, because $\alpha_0>0$ so that either $f \neq 0$ or $v \neq 0$ holds.
On the other hand, we can write 
\begin{align*}
    \mathcal{V}_{k_j} u_j - \mathcal{V}_{k_*} u
    =
    (\mathcal{V}_{k_j} - \mathcal{V}_{k_*})u_j + \mathcal{V}_{k_*}(u_j-u),
\end{align*}
and by Lemma~\ref{lemma:Ak-continuity} and Corollary~\ref{corollary:zetak-continuity}, 
$\mathcal{V}_{k_j} u_j$ converges strongly to $\mathcal{V}_{k_*} u$ in $H^{-\frac{1}{2}}\oplus \mathbb{C}^N$.
This implies $\mathcal{V}_{k_*} u= 0$ hence $u=0$ by Proposition~\ref{proposition:vector-well-posed}, a contradiction.
\end{proof}

\newpage
\section{High-Order Discretization}
$L^2$-stability allows flexibility in basis choice.
We want to employ at least piecewise polynomial curvilinear elements to represent $\Gamma$,
but this leads to a significant numerical challenge in evaluating $S_k \textnormal{curl}_\Gamma$ and $\textnormal{curl}_\Gamma S_k$. 
In their direct form, these two terms are expressed as Cauchy principal value integrals, 
where the limit depends on the eigenvalues of the Jacobian at the point of evaluation.
These complications make it difficult to evaluate them both quickly and systematically.
Let us write the scattering surface $\Gamma$ as a union 
\begin{align*}
    \Gamma = \cup_{j=1}^{n_U} \overline{U_j},
\end{align*}
where the $U_j \subset \Gamma$ are open, disjoint, having piecewise smooth boundaries in common.
In the vector case, we must deal with the aforementioned problem in one of two ways.
Let $a,\psi \in L^2(\Gamma, T_\mathbb{C}\Gamma)$ and $\rho, \varphi\in L^2(\Gamma)$ have $C^\infty(U_j)\cap C^0(\overline{U_j})$ representatives on $U_j$.
Using integration by parts, we compute
\begin{align*}
    \langle S_k\textnormal{curl}_\Gamma a, \varphi \rangle
    =
    \sum_{j=1}^{n_U}\Big[ \int_{U_j} (\textnormal{curl}_{U_j} a|_{U_j}) S_k\varphi \ d\sigma  + \int_{\partial U_j} (t_{\partial U_j} \cdot a|_{U_j}) S_k \varphi \ dl \Big],
\end{align*}
and
\begin{align*}
    \langle \textnormal{curl}_\Gamma S_k \rho, \psi \rangle
    =
    \sum_{j=1}^{n_U}\Big[ \int_{U_j} (\textnormal{curl}_{U_j} \psi|_{U_j}) S_k\rho \ d\sigma  + \int_{\partial U_j} (t_{\partial U_j} \cdot \psi|_{U_j}) S_k \rho \ dl \Big],
\end{align*}
where $t_{\partial U_j}$ is the unit tangent field on $\partial U_j$ oriented counter-clockwise relative to $\nu$. 
Note that $S_k \rho|_{U_j}$ and $S_k \varphi|_{U_j}$ are in $H^1(U_j)$ and their traces on $\partial U_j$ are in $H^{+\frac{1}{2}}(\partial U_j)$, 
but with $\partial U_j$ a smooth manifold with corners (piecewise smooth closed curve) in $\Gamma$. 
The identities are thus obtained from the smooth case by a simple density argument.
Consequently, either we evaluate the line integrals, or we cancel them out on edges.
The latter option is practically easier, amounting to tangential continuity across edges, 
but the former achieves superior conditioning through $L^2$-orthonormal basis functions. \\

One reasonable way to approach discretization is to fix an atlas once and for all.
Refinement is in the parameter domain, geometric dependence isolated in constants.
In this setup, we assume each piece is parametrized by a smooth map
\begin{align*}
    \tau_j : [-1, +1]^2 \to \overline{U_j},
\end{align*}
which extends smoothly to a diffeomorphism on a neighborhood of $[-1, +1]^2$ in $\mathbb{R}^2$.
This ensures that its surface Jacobian is bounded from below by a positive constant, 
and it is of course always possible to arrange this by refining a suitable finite atlas.
The above observations on $S_k \textnormal{curl}_\Gamma$ and $\textnormal{curl}_\Gamma S_k $ transfer to all subdivisions of the $U_j$.
Define for $j$ a sequence in $n\in \mathbb{N}$ of partitions $\mathcal{Q}_{j,n}$ of $[-1,+1]^2$ into quadrangles $Q$, 
where each $Q$ carries a positively oriented affine transformation $R_Q : [-1,+1]^2 \to Q$. 
It is convenient to write
\begin{align*}
    h(Q) = \textnormal{diam}(Q),
\end{align*}
and $Q\in \mathcal{Q}_{j,n}$ is assigned a number $p(Q) \geq 1$, the tensor product approximation order.
Now let $\{P_n\}_{n=0}^\infty \subset L^2(-1,+1)$ be the standard orthonormal Legendre polynomials. 
The idea is to build, and pull back to $\Gamma$, the tensor product polynomial on each $Q$. 
Convergence is then controlled locally by the Babu\v{s}ka-Suri estimate \cite[Lemma 4.5]{BabuskaSuri1987}.

\newpage
Consider a fixed quadrangle $Q\in \mathcal{Q}_{j,n}$ in the $n$-th partition of the domain of $\tau_j$.
Associated to $Q$, we have the affine map $R_Q$ and the total polynomial order $p=p(Q)$.
In the scalar case, on $\tau_j(Q)$ we choose basis functions
\begin{align*}
    \rho_{m,n} = ((\tau_j \circ R_Q)^{-1})^*\Big(\frac{1}{\sqrt{J_j}}P_m \boxtimes P_n \Big),
\end{align*}
where $J_j$ is the surface Jacobian of $\tau_j \circ R_Q : [-1,1]^2 \to \tau_j(Q)$ in the set $[-1,+1]^2$.
These basis functions are by construction $L^2(\Gamma)$-orthonormal due to the factor $1/\sqrt{J_j}$.
In the vector case, we might try
\begin{align*}
    a_1^{m,n} &= \rho_{m,n} ( \Lambda_{11} e_1 + \Lambda_{12} e_2 ), \\
    a_2^{m,n} &= \rho_{m,n} ( \Lambda_{21} e_1 + \Lambda_{22} e_2),
\end{align*}
where $\{e_1, e_2\}$ is the local frame of $T \tau_j(Q)$ induced by $\tau_j \circ R_Q$, and
\begin{align*}
\begin{bmatrix}
    \Lambda_{11} & \Lambda_{12} \\
    \Lambda_{21} & \Lambda_{22}
\end{bmatrix}
&=
\begin{bmatrix}
    e_1 \cdot e_1 & e_1 \cdot e_2 \\
    e_2 \cdot e_1 & e_2 \cdot e_2
\end{bmatrix}^{-\frac{1}{2}} \\
&= 
\frac{\sqrt{(J_j + |e_1|^2)+(J_j + |e_2|^2)}}{(J_j + |e_1|^2)(J_j + |e_2|^2) - (e_1 \cdot e_2)^2}
\begin{bmatrix}
    |e_2|^2 + J_j & -e_1 \cdot e_2 \\
    -e_1 \cdot e_2 & |e_1|^2 + J_j
\end{bmatrix},
\end{align*}
which is to say that the matrix $\Lambda$ is just the square root of the inverse metric tensor.
A basis expansion on $\tau_j(Q)$ is of the form
\begin{align*}
    a =   \sum_{m=0}^{p} \sum_{n=0}^{p} C_{m,n}^1 a^{m,n}_1 + \sum_{m=0}^{p} \sum_{n=0}^{p} C_{m,n}^2 a^{m,n}_2.
\end{align*}
This construction yields an orthonormal system, at the price of geometric complexity.
Alternatively,  
we can use the modified polynomials $\{ \widetilde{P}_n \}_{n=0}$ of \cite{Jorgensen2003} given by
\begin{align*}
    \widetilde{P}_n(x)
    = 
    \begin{cases}
        1 - x & \textnormal{if} \quad n = 0, \\
        1 + x & \textnormal{if} \quad n = 1, \\
        P_n(x) - P_{n-2}(x) & \textnormal{if} \quad n \geq 2,
    \end{cases}
\end{align*}
and the contravariant basis
\begin{align*}
    a^1_{m,n} = ((\tau_j \circ R_Q)^{-1})^*(P_m \boxtimes \widetilde{P}_n) e^1, \\
    a^2_{m,n} = ((\tau_j \circ R_Q)^{-1})^*(\widetilde{P}_m \boxtimes P_n) e^2,
\end{align*}
where $\{ e^1, e^2 \}$ is the dual frame to $\{e_1, e_2\}$, a frame for the cotangent bundle $T^*\tau_j(Q)$.
Another basis expansion on $\tau_j(Q)$ is then of the form
\begin{align*}
    a^\flat =   \sum_{m=0}^{p} \sum_{n=0}^{p+1} C_1^{m,n} a^1_{m,n} + \sum_{m=0}^{p+1} \sum_{n=0}^{p} C^{m,n}_2 a^2_{m,n},
\end{align*}
and on any two quadrangles with a shared edge tangential continuity is easy to enforce. 
It amounts to equal coefficients for the adjacent basis functions non-zero on the edge.
Note that the sums are made to fit with the Nedelec constraint for the surface curl, 
which must be computed for either choice. In the latter case, it is much simpler \cite{Jorgensen2003}.

\newpage
\begin{corollary}\label{corollary:hp-convergence}
Let the mesh be shape-regular, in the sense of \cite[pp. 203, 3.5c]{BabuskaSuri1987}.
On a given quadrangle $Q \in \mathcal{Q}_{j,n}$ we write
\begin{align*}
   \varepsilon(Q) =  C_m\frac{h(Q)^{\min\{p(Q)+1, m \}}}{p(Q)^m},
\end{align*}
where $C_m>0$ depends only on $m$, the geometry of $\Gamma$, and the basis. Also, let $n\geq n_0$.
Writing $\mathcal{V}^{-1}_kg = (f,v)$ or $\mathcal{S}^{-1}_kg = (f,v)$
for $k\in [0,K]$ in the discontinuous setting:

\medskip
\begin{enumerate}
    \item If $g\in L^2(\Gamma,T_\mathbb{C}\Gamma \oplus \mathbb{C})\oplus \mathbb{C}^N$ and $u_n$ solves the discrete problem for $\mathcal{V}_k$, then
\begin{align*}
    || \mathcal{V}^{-1}_kg - u_n ||_{L^2(\Gamma,T\Gamma \oplus \mathbb{C})\oplus \mathbb{C}^N}  \leq \frac{||\mathcal{V}_k||}{\alpha_K} \sum_{j=1}^{n_U} \sum_{Q\in \mathcal{Q}_{j,n}}  \varepsilon(Q) ||f||_{H^{m}(\tau_j(Q),T_\mathbb{C}\tau_j(Q) \oplus \mathbb{C})}.
\end{align*}
    \item If $g\in L^2(\Gamma)\oplus \mathbb{C}^N$ and $u_n$ solves the discrete problem for $\mathcal{S}_k$, then
\begin{align*}
    || \mathcal{S}^{-1}_kg - u_n ||_{L^2(\Gamma)\oplus \mathbb{C}^N}  \leq \frac{||\mathcal{S}_k||}{\beta_K} 
    \sum_{j=1}^{n_U} \sum_{Q\in \mathcal{Q}_{j,n}}  \varepsilon(Q) ||f||_{H^m(\tau_j(Q))}.
\end{align*}
\end{enumerate}
\end{corollary}
\begin{proof}
This is quasi-optimality and the Babu\v{s}ka-Suri estimate \cite[Lemma 4.5]{BabuskaSuri1987}.
Put $H = L^2(\Gamma,T_\mathbb{C}\Gamma \oplus \mathbb{C})\oplus \mathbb{C}^N$. Let $\mathcal{P}_{p}$ be the polynomials up to order $p$ on $(-1,+1)$.
Allowing $H_n$ to be the spaces formed by $\{\mathcal{Q}_{j,n}\}$ in the vector case, we get
\begin{align*}
    \inf_{w\in H_n} || \mathcal{V}^{-1}_kg - w||_{H}
    &\leq
    \inf_{w\in H_n} \sum_{j=1}^{n_U} || (f, 0) - w||_{L^2(U_j; T_\mathbb{C}U_j \oplus \mathbb{C})\oplus \mathbb{C}^N} \\
    &\leq
    \sum_{j=1}^{n_U} \sum_{Q\in \mathcal{Q}_{j,n}} \inf_{v \in P_Q} \Big|\Big| \tau_j^*f - (\tau_j^*\Lambda \oplus 1) (R_Q^{-1})^*\Big(\frac{1}{\sqrt{J_j}}v\Big) \Big|\Big|_{L^2(Q; \mathbb{C}^3)} \\
    &\leq
    \sum_{j=1}^{n_U} \sum_{Q\in \mathcal{Q}_{j,n}} C_j'\inf_{v \in P_Q} \Big|\Big| \sqrt{J_j \circ R_Q^{-1}}(\tau_j^*\Lambda \oplus 1)^{-1} \tau_j^*f - v \Big|\Big|_{L^2(Q; \mathbb{C}^3)} \\
    &\leq
    \sum_{j=1}^{n_U} \sum_{Q\in \mathcal{Q}_{j,n}}  \varepsilon(Q) || \tau_j^*f ||_{H^m(Q; \mathbb{C}^3)} \\
    &=
    \sum_{j=1}^{n_U} \sum_{Q\in \mathcal{Q}_{j,n}}  \varepsilon(Q) || f ||_{H^m(\tau_j(Q); T_\mathbb{C}\tau_j(Q) \oplus \mathbb{C})},
\end{align*}
where $P_Q = (\mathcal{P}_{p(Q)} \boxtimes \mathcal{P}_{p(Q)})^3$ and we have $\mathcal{V}^{-1}_kg = (f, v)$
with $f \in  L^2(\Gamma,T_\mathbb{C}\Gamma \oplus \mathbb{C})$.
Note that $H_n$ always has the finite $\mathbb{C}^N$ summand, so the first inequality makes sense.
Also, the constants $C_j'>0$ are here the intermediary quantities
\begin{align*}
  C_j' =  \max_{Q} \frac{||(\tau_j^*\Lambda \oplus 1) ||}{\sqrt{J_j\circ R_Q^{-1}}} 
\end{align*}
which are fixed once and for all by geometry, and are independent of everything else.
The argument is similar in the scalar case.
\end{proof}

Note that only shape-regularity, not quasi-uniformity, is required for the estimates.
Also, in the tangentially continuous case, a Scott-Zhang type operator can be built, 
and a slightly modified estimate can similarly be proved for the vector part.

\newpage
Finally, we obtained stability and convergence for \textit{exact} geometry representation.
But in practice, we must defer to an imperfect patch-wise polynomial approximation,
and so, technically speaking, the obtained results are no longer valid in that setting.
Nevertheless, similar results \textit{should} still hold if the representation is "good enough",
where the exact surface is replaced by a continuous, piece-wise polynomial surface.
This leap of faith really demands proper justification, but we shall not do this here.
Concretely, we replace the $\tau_j\circ R_Q$ by patches of the form
\begin{align*}
    \widehat{\tau} : [-1, +1]^2 \to \mathbb{R}^3 : (u,v) \mapsto \sum_{k=0}^q \sum_{l=0}^q r_{k,l} (\ell_k^q \boxtimes\ell^q_l)(u,v),
\end{align*}
where $\ell^p_k$ is the Lagrange polynomial
\begin{align*}
    \ell_k^q(u) = \prod_{k\neq l}^q \frac{u - u_l}{u_k - u_l},
\end{align*}
and $r_{k,l} = (\tau_j\circ R_Q)(t_k,t_l)$ are interpolation nodes with $t_k = -1 + \frac{2k}{q}$ equidistant.
Note that $q$ is the patch order. It is not coupled to the order(s) of the basis functions.
Another way to write this is in the monomial basis
\begin{align*}
    \widehat{\tau}(u,v) = \sum_{k=0}^q \sum_{l=0}^q \widehat{r}_{k, l} u^k v^l,
\end{align*}
which is more suitable for a general numerical implementation as a recursive scheme. 
Also, the coefficients $\widehat{r}_{k, l}$ can be expressed directly in terms of the interpolation nodes, and a small table recording this for $q = 1$ and $q=2$ is found in \cite[Appendix B.1]{Jorgensen2003}.
In the next section, we shall employ patches of order $q=4$ only, requiring $25$ nodes.\\

We shall refer to the tangentially continuous/discontinuous schemes as TCG/DG.
The latter is necessarily going to be spectrally convergent, in view of Corollary~\ref{corollary:hp-convergence},
and, as noted, a parallel argument leads to the same conclusion for the TCG DPIE.
The matrix filling time is expected to be much longer compared to EFIE/MFIE/CFIE, 
because the DPIE system matrices are far more complicated and involve more terms. 
But their complexities are necessarily the same, differing only by a constant overhead,
where the bottleneck is just evaluation of $\Phi_k$ once in each patch-patch interaction.
This is good news, because large problems are usually dominated by the solve step, 
which is performed by an (accelerated) iterative method like GMRES or BiCGstab.
The amount of time spent in this step is linked to the system matrix condition number, 
which controls how many iterations, hence matrix-vector multiplications, are required.
But the conditioning of the Galerkin schemes here are independent of the element size,
because of the operator level immunity to low-frequency and topological breakdown,
and our high-order polynomial implementations do not destroy this desirable property.
\newline
In our implementation, it is expected to grow only mildly with the polynomial order,
and should mirror the CFIE discretized with a high-order divergence-conforming basis.
(A true independence of order is really only achievable for orthogonal basis choices.)
It is thus expected that for large multiscale problems, the DPIE will perform better.
Competing formulations, such as regularized CFIE \cite{Bruno2009}, analytic preconditioning \cite{Baumann2023}, and the method of generalized Debye sources \cite{EpsteinGreengard2010Debye, EpsteinGreengardONeil2013Debye} seek to attain similar advantages.
Also, in many applied cases, e.g. radar cross section, only the far-field is of interest.
In that situation, one can do away with the scalar equation altogether.


\newpage
\section{Conditioning and Convergence}
Provided $k\geq 1$, we rescale as in \cite{VicoGreengard2016, VicoGreengard2025}.
Originally due to \cite{Kress1985}, this should result in optimal conditioning at high frequencies.
The vector DPIE is then for $k\geq 1$ replaced by its scaled version
\begin{align*}
\Big(
\begin{bmatrix}
    \frac{1}{2}I + M_k & k L_k \\
    0 &  \frac{1}{2}I + D_k
\end{bmatrix}
+
i\eta
\begin{bmatrix}
   k JS_kJ & \textnormal{curl}_\Gamma S_k  \\
   S_k \textnormal{curl}_\Gamma & -k S_k
\end{bmatrix}
\Big)
\begin{bmatrix}
    a \\
    \rho
\end{bmatrix}
-
\begin{bmatrix}
    0 \\
    \sum_{j=1}^N v_j \chi_{\Gamma_j}
\end{bmatrix}
&=
\begin{bmatrix}
    f \\
   \frac{1}{k} h
\end{bmatrix}, \\
\int_{\Gamma_j} - k \nu \cdot S_k (\nu\rho) + i\eta \Big( k \nu \cdot S_k Ja -\frac{1}{2}\rho + D^t_k \rho \Big) \, d\sigma &= q_j,
\end{align*}
and the scalar DPIE is similarly replaced by
\begin{align*}
\Big(
\frac{1}{2}I + D_k
-
i\eta k
S_k 
\Big)
\rho
- \sum_{j=1}^N v_j \chi_{\Gamma_j}
&=
h, \\
\int_{\Gamma_j} (W_k - W_0)\rho + i\eta k\Big(\frac{1}{2} - D^t_k \Big)\rho \, d\sigma &=  q_j,
\end{align*}
where the scaling balances the extra norm-growth in $k$ due to the derivatives on $S_k$.
It is clear that these scalings have no influence whatsoever on Theorems~\ref{theorem:DPIE-well-posed} and \ref{theorem:DPIE-galerkin-stable}.
The element approximation results of the previous section remain unchanged as well. \\

All the elements are now in place for studying the properties of our discretization.
The scattered field in the unbounded medium is determined by 
\begin{align*}
\phi &= D_k\rho - i\eta S_k \rho, \\
A &=  \nabla \times S_k a -  S_k(\nu \rho) + i \eta( S_k(\nu \times a) + \nabla S_k\rho),
\end{align*}
and in the $k\geq 1$ optimally scaled case by
\begin{align*}
\phi &= D_k\rho - i\eta k S_k \rho, \\
A &=  \nabla \times S_ka - k S_k(\nu \rho) + i \eta(k S_k(\nu \times a) + \nabla S_k\rho),
\end{align*}
where $S_k$ and $D_k$ by abuse of notation denote the corresponding potentials in $\mathbb{R}^3 \setminus \Gamma$.
These fields are in the Lorenz gauge by definition, as checked by a direct computation.
It is always assumed that the incident fields $\phi^\textnormal{inc}$ and $A^\textnormal{inc}$ are also in the Lorenz gauge.
Now in the vector case, the boundary conditions require that
\begin{align*}
    f &= - \nu \times A^\textnormal{inc}|_\Gamma, \\
    h &= - \nabla \cdot A^\textnormal{inc}|_{\Gamma}, \\
    q_j &= - \int_{\Gamma_j} \nu \cdot A^\textnormal{inc} \, d\sigma,
\end{align*}
whereas in the scalar case, we must impose
\begin{align*}
    h &= - \phi^\textnormal{inc}|_\Gamma, \\
    q_j &= - \int_{\Gamma_j} \nu \cdot \nabla \phi^\textnormal{inc} \, d\sigma,
\end{align*}
and in both cases the scattered fields are bounded as $k\to 0$ if the incident fields are.
Moreover, the scattered fields are then also in the Lorenz gauge \cite[Theorem 4.2]{VicoGreengard2016}.

\newpage
A natural choice of trial incident field is a plane wave in $\mathbb{R}^3$ with direction $\hat{k}\in \mathbb{S}^2$.
Given a fixed polarization vector $E_p\in \mathbb{C}^3$, we represent it by the potentials
\begin{align*}
\phi^\textnormal{inc} &: \mathbb{R}^3 \to \mathbb{C} : x \mapsto -x\cdot E_p e^{i k \hat{k}\cdot x}, \\
A^\textnormal{inc} &: \mathbb{R}^3 \to \mathbb{C}^3 : x \mapsto -(x\cdot E_p) c_0^{-1} e^{i k \hat{k}\cdot x} \hat{k},
\end{align*}
which are related via the Lorenz gauge, and both potentials remain bounded as $k\to 0$.
Here $c_0$ is the speed of light in vacuum. We fix vacuum as the medium, and fix $\eta = 1$.
Our incident plane wave is linearly polarized, so that $E_p = (1,0,0)$, and $\hat{k} = (0, 0, -1)$. \\

\begin{figure}[H]
    \centering
    \begin{subfigure}{0.48\textwidth}
        \centering
        \includegraphics[width=0.75\textwidth]{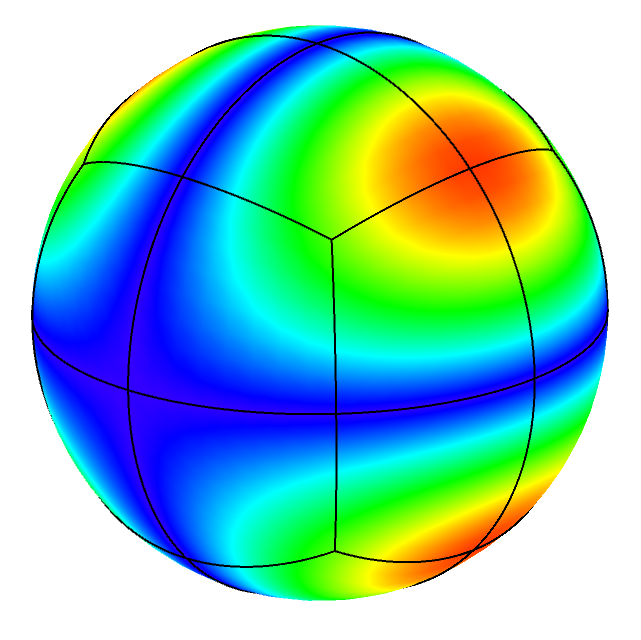}
        \caption{$|a_1|$ of $a=(a_1,a_2,a_3)$.}
    \end{subfigure}
    \hfill
    \begin{subfigure}{0.48\textwidth}
        \centering
        \includegraphics[width=0.75\textwidth]{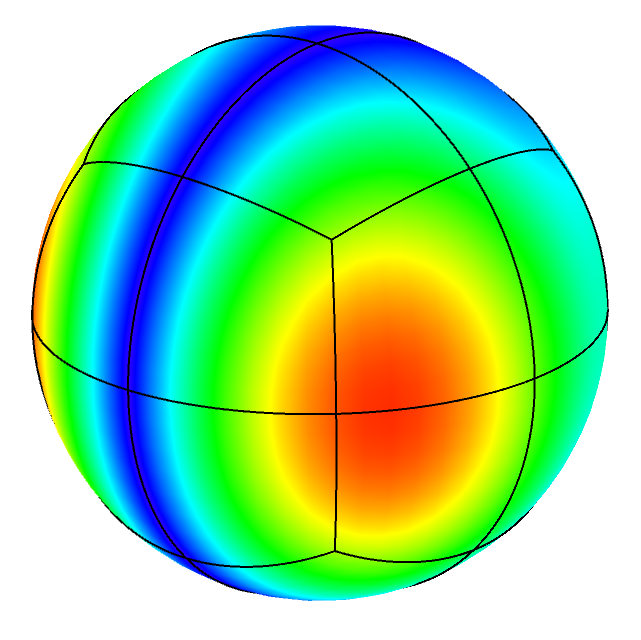}
        \caption{$|\rho|$.}
    \end{subfigure}
    \caption{Unit sphere vector DPIE solution at $k=\pi/8$.}
    \label{figure:dpie_sphere}
\end{figure}

\noindent
Above, in Figure~\ref{figure:dpie_sphere}, the basis is discontinuous with $p=6$, all patches are fourth-order.
It is also interesting to test a scatterer of non-zero genus, with multiple components.
Here we simply take a link of two tori, both of inner radius $0.15$ and outer radius $0.5$.

\begin{figure}[H]
    \centering
    \begin{subfigure}{0.48\textwidth}
        \centering
        \includegraphics[width=0.75\textwidth]{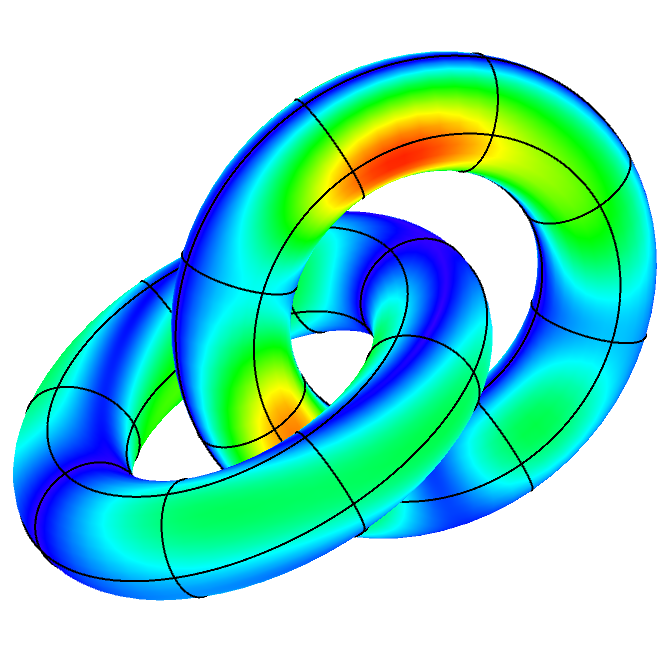}
        \caption{$|a_1|$ of $a=(a_1,a_2,a_3)$.}
    \end{subfigure}
    \hfill
    \begin{subfigure}{0.48\textwidth}
        \centering
        \includegraphics[width=0.75\textwidth]{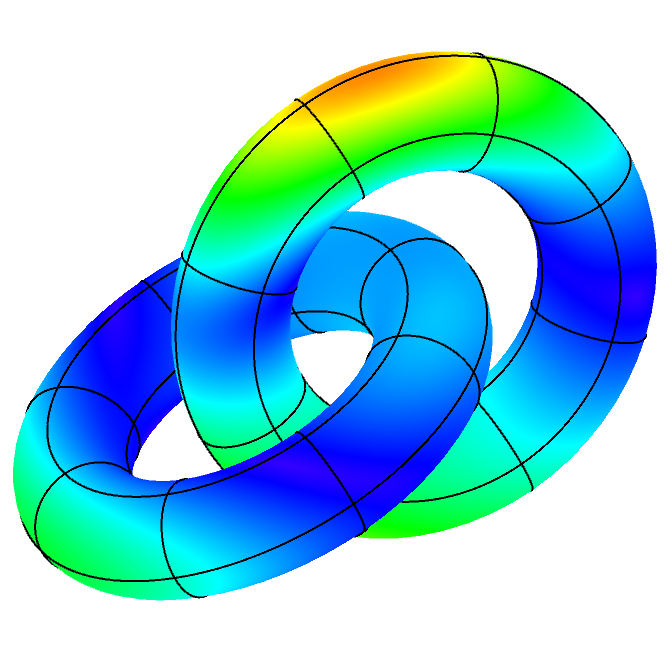}
        \caption{$|\rho|$.}
    \end{subfigure}
    \caption{Linked tori vector DPIE solution at $k=\pi/2$.}
    \label{figure:dpie_linked_tori}
\end{figure}

\noindent
Again, in Figure~\ref{figure:dpie_linked_tori}, the numerical settings are the same as for the above unit sphere.

\newpage
Once appropriately scaled, the schemes stabilize the condition number very well.
In fact, the fully discontinuous choice appears to eliminate dependence on the order,
whereas enforcing tangential continuity causes only mild growth upon order increase.
Concretely, if $Z$ is the Galerkin matrix in a basis of $H_n$ and $G$ its Gram matrix
\begin{align*}
    \kappa(Z) \leq \Big(\alpha_K^{-1}\sup_{k\in [0,K]}||\mathcal{V}_k||_{B(H^0\oplus \mathbb{C}^N)} \Big)  \kappa(G)
    \quad
    \textnormal{if}
    \quad
    n\geq n_0.
\end{align*}
A clear demonstration of this growth bound appears in Figure~\ref{figure:dpie_sphere_conditioning}.\\

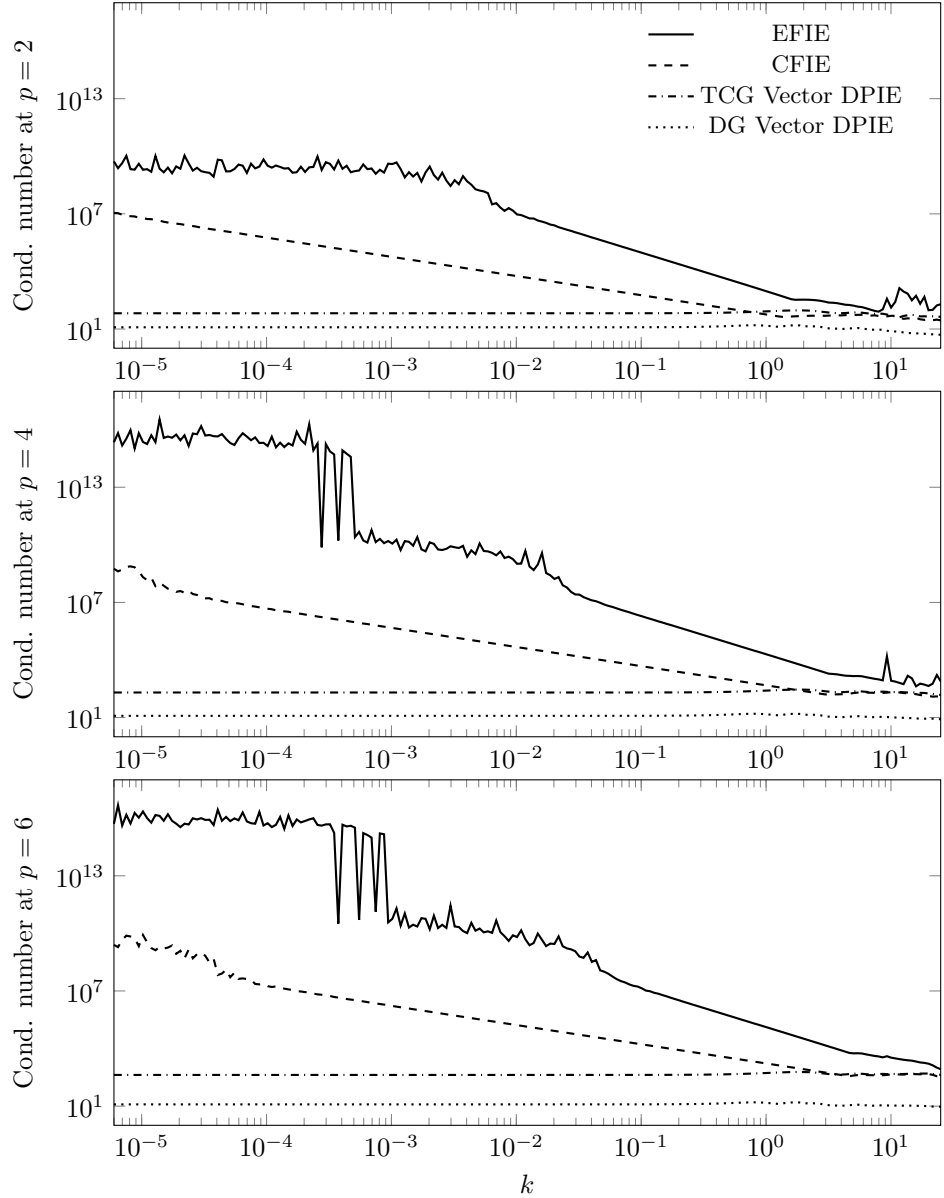
\begin{figure}[H]
    \centering
    \begin{tikzpicture}
        \begin{loglogaxis}[
            width = 0.96\textwidth,
            height = 5.9cm,
            xmin=6e-6,
            xmax=2.5e1,
            ymin=1e0,
            ymax=1e18,
            ylabel={Cond. number at $p=2$},
            legend pos=north east,
            legend style={
                draw=none,
                fill=none,
                font=\small
            },
        ]
        
        \addplot[solid, thick]
        table[x = k, y = c_efie_order_2]{condition.dat};
        \addlegendentry{EFIE}
        
        \addplot[dashed, thick]
        table[x = k, y = c_cfie_order_2]{condition.dat};
        \addlegendentry{CFIE}

        \addplot[dashdotted, thick]
        table[x = k, y = c_vdpie_order_2]{condition.dat};
        \addlegendentry{TCG Vector DPIE}

        \addplot[dotted, thick]
        table[x = k, y = c_vdpie_break_order_2]{condition.dat};
        \addlegendentry{DG Vector DPIE}
        
        \end{loglogaxis}
    \end{tikzpicture}
    \begin{tikzpicture}
        \begin{loglogaxis}[
            width = 0.96\textwidth,
            height = 5.9cm,
            xmin=6e-6,
            xmax=2.5e1,
            ymin=1e0,
            ymax=1e18,
            ylabel={Cond. number at $p=4$},
            legend pos=north east,
            legend style={
                draw=none,
                fill=none,
                font=\small
            },
        ]
        
        \addplot[solid, thick]
        table[x = k, y = c_efie_order_4]{condition.dat};
        
        \addplot[dashed, thick]
        table[x = k, y = c_cfie_order_4]{condition.dat};

        \addplot[dashdotted, thick]
        table[x = k, y = c_vdpie_order_4]{condition.dat};

        \addplot[dotted, thick]
        table[x = k, y = c_vdpie_break_order_4]{condition.dat};
        
        \end{loglogaxis}
    \end{tikzpicture}
    \begin{tikzpicture}
        \begin{loglogaxis}[
            width = 0.96\textwidth,
            height = 5.9cm,
            xmin=6e-6,
            xmax=2.5e1,
            ymin=1e0,
            ymax=1e18,
            xlabel={$k$},
            ylabel={Cond. number at $p=6$},
            legend pos=north east,
            legend style={
                draw=none,
                fill=none,
                font=\small
            },
        ]
        
        \addplot[solid, thick]
        table[x = k, y = c_efie_order_6]{condition.dat};
        
        \addplot[dashed, thick]
        table[x = k, y = c_cfie_order_6]{condition.dat};

        \addplot[dashdotted, thick]
        table[x = k, y = c_vdpie_order_6]{condition.dat};

        \addplot[dotted, thick]
        table[x = k, y = c_vdpie_break_order_6]{condition.dat};
        
        \end{loglogaxis}
    \end{tikzpicture}
    \caption{Unit sphere model condition numbers due to the mesh shown in Figure~\ref{figure:dpie_sphere}. 
    Comparison is made with a commercial solver using divergence-conforming bases \cite{TicraSoftware}.
    Note the spikes due to the presence of resonances in the EFIE for $k>2$.
    }
    \label{figure:dpie_sphere_conditioning}
\end{figure}

\newpage
Clearly, it achieves better conditioning than the commercial EFIE/CFIE solver.
In the high-frequency regime, the fully discontinuous version even outperforms both.
The total matrix filling time is still dominated by evaluations of the Helmholtz kernel, 
which is the complexity scaling bottleneck, making it viable for multi-scale problems.\\
An $(h,p)$-convergence check versus the analytic Mie series is shown below in Figure~\ref{figure:dpie_hp_sphere}.\\

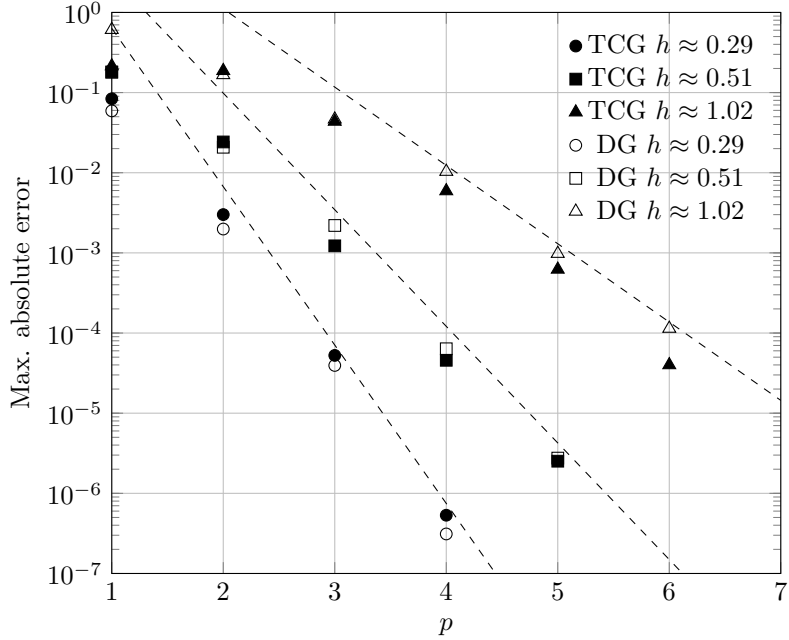
\begin{figure}[H]
    \centering
    \begin{tikzpicture}
        \begin{semilogyaxis}[
            width = 0.8\textwidth,
            height = 9cm,
            ylabel = {Max. absolute error},
            xlabel = {$p$},
            xmin = 1,
            xmax = 7,
            ymin = 1e-7,
            ymax = 1,
            xtick = {1,2,3,4,5,6,7},
            grid = major,
            legend style={
                draw = none,
                fill = none,
                at = {(0.98,0.98)},
                anchor = north east,
            },
        ]

        \addplot[
            black,
            only marks,
            mark = *,
            mark size = 2.2pt,
            restrict x to domain = 1:4,
        ]
        table[x=p,y=m025]{tcg_hp_sphere.dat};
        \addlegendentry{TCG $h \approx 0.29$}
        
        \addplot[
            black,
            only marks,
            mark = square*,
            mark size = 2.2pt,
            restrict x to domain = 1:5,
        ]
        table[x=p,y=m05]{tcg_hp_sphere.dat};
        \addlegendentry{TCG $h \approx 0.51$}
        
        \addplot[
            black,
            only marks,
            mark = triangle*,
            mark size = 2.8pt,
        ]
        table[x=p,y=m1]{tcg_hp_sphere.dat};
        \addlegendentry{TCG $h \approx 1.02$}

        \addplot[
            black,
            only marks,
            mark = o,
            mark size = 2.2pt,
            restrict x to domain = 1:4,
        ]
        table[x=p,y=m025]{dg_hp_sphere.dat};
        \addlegendentry{DG $h \approx 0.29$}
        
        \addplot[
            black,
            only marks,
            mark = square,
            mark size = 2.2pt,
            restrict x to domain = 1:5,
        ]
        table[x=p,y=m05]{dg_hp_sphere.dat};
        \addlegendentry{DG $h \approx 0.51$}
        
        \addplot[
            black,
            only marks,
            mark = triangle,
            mark size = 2.8pt,
        ]
        table[x=p,y=m1]{dg_hp_sphere.dat};
        \addlegendentry{DG $h \approx 1.02$}

        \addplot[dashed, black, domain=1:7] {1e2*exp(-2.25*x)};
        \addplot[dashed, black, domain=1:7] {8e1*exp(-3.35*x)};
        \addplot[dashed, black, domain=1:7] {6e1*exp(-4.55*x)};
        
        \end{semilogyaxis}
    \end{tikzpicture}
    \caption{Unit sphere scattered electric near-field versus analytic Mie series at $k=2\pi$. The fields are compared on a circle in the $xz$-plane of radius $2$ centered at the origin, and both fields are numerically evaluated in exactly $361$ arc-length equidistant points.}
    \label{figure:dpie_hp_sphere}
\end{figure}

\begin{figure}[H]
    \centering
    \begin{subfigure}{0.30\textwidth}
        \centering
        \includegraphics[width=0.5\textwidth]{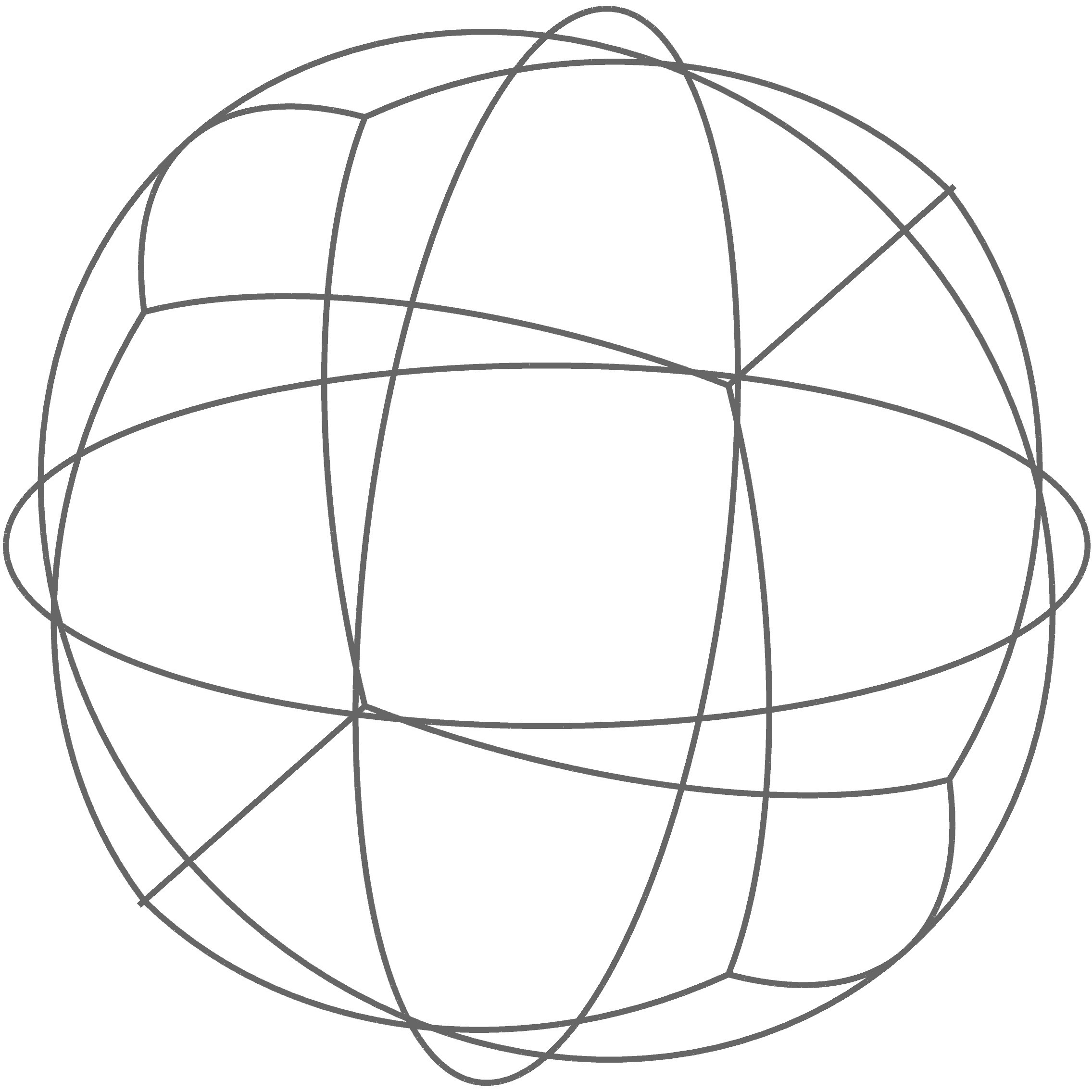}
        \caption{$h \approx 1.02$.}
    \end{subfigure}
    \hfill
    \begin{subfigure}{0.30\textwidth}
        \centering
        \includegraphics[width=0.5\textwidth]{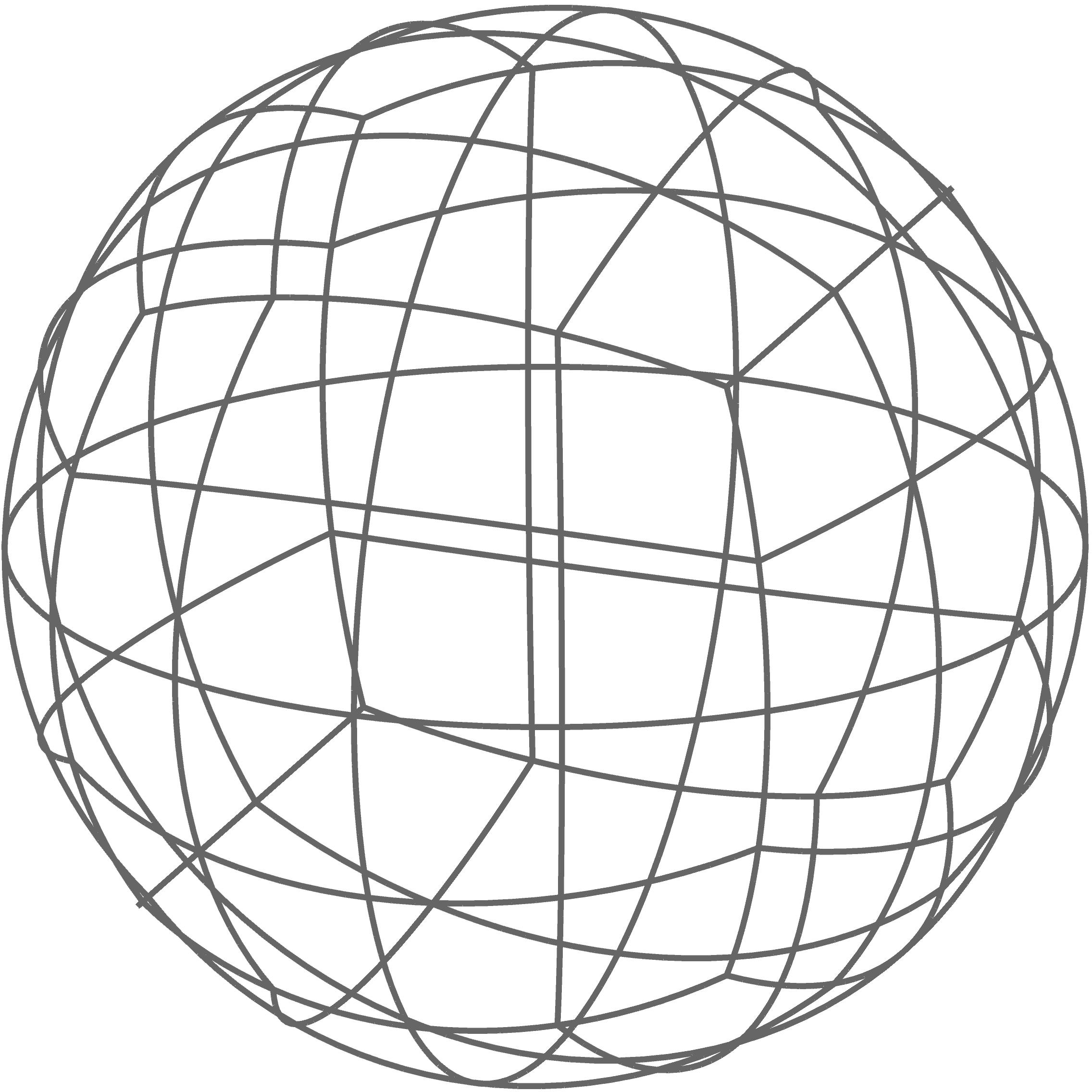}
        \caption{$h \approx 0.51$.}
    \end{subfigure}
    \hfill
    \begin{subfigure}{0.30\textwidth}
        \centering
        \includegraphics[width=0.5\textwidth]{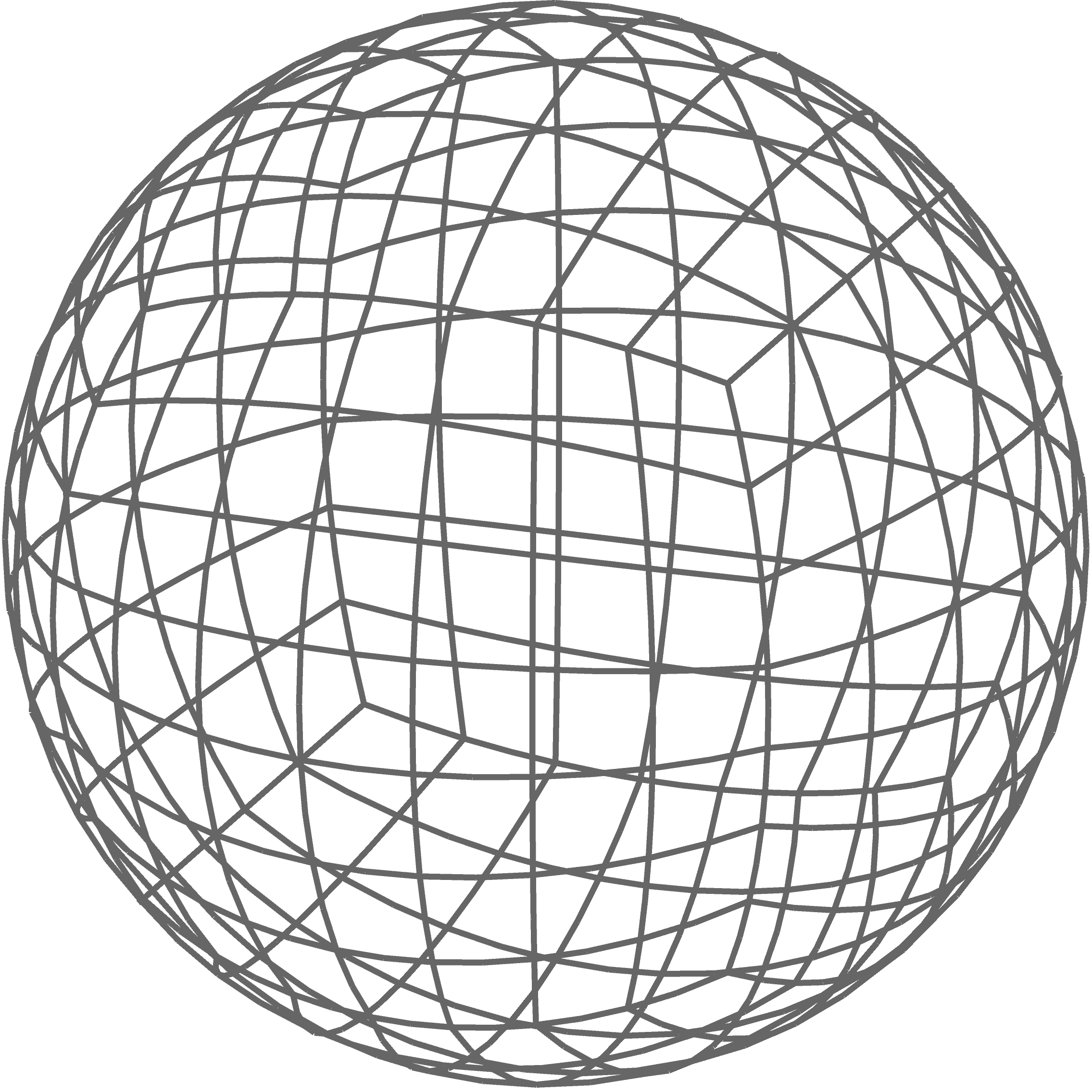}
        \caption{$h \approx 0.29$.}
    \end{subfigure}
    \caption{Quadrilateral patches with average diameter $h$ used in Figure~\ref{figure:dpie_hp_sphere}.}
    \label{figure:hp_sphere}
\end{figure}

Of course, realistically, the scatterer model surface $\Gamma$ is only piece-wise smooth.
Further experiments hint that tangential continuity is the right constraint in that case,
and that the solver is convergent, if this is imposed at non-smooth edges and vertices.
This suggests that the correct approach is a hybrid continuous-discontinuous scheme,
where mesh flexibility and near-orthogonality are preserved to the best possible extent.
However, we leave this broader study for a future paper.


\newpage
\section{Appendix: Principal Symbols}
Theorem~\ref{theorem:layer-psido} is implicit in the literature.
It is well-known that all the layer potential operators are classical pseudo-differential, 
but there seems to be a lack of clarity about their orders and precise principal symbols.
Sometimes, the double layer operator is understood as a limit from the exterior to $\Gamma$,
which is then $\frac{1}{2}I+D_k$, because of the jump relations, and obviously not of order $-1$.
We have therefore decided to include a proof of the correct orders here.

\begin{proof}[Proof of Theorem~\ref{theorem:layer-psido}]
Let $\mathbb{D}$ be the closed unit disc in $\mathbb{R}^2$. Fix $(x, \xi) \in T^* \Gamma \setminus 0$.
Pick a conformal parametrization $\tau : \mathbb{D} \to \Gamma$ with $\tau(0) = x$ and conformal factor $\mu^2$, 
and let $\varphi = \xi^\sharp \cdot \tau^{-1}$ be a phase function with $d\varphi(x)=\xi$ on the interior set of $\tau(\mathbb{D})$.
It is tacitly extended using a cutoff $\chi \in C^\infty_0(\tau(\mathbb{D}^\circ))$ that is $1$ in a neighborhood of $x$, 
which is defined by $\chi = \chi_0 \circ \tau^{-1}$
for a symmetric $\chi_0 \in C^\infty_0(\mathbb{D}^\circ)$ equal to $1$ near $0$.
Working in this setup, we can directly compute
\begin{align*}
e^{-i \lambda \varphi(x)} S_k( e^{i \lambda \varphi} \chi )(x)
&=
\int_\Gamma \Phi_k(x,y) e^{i \lambda \varphi(y)} \chi(y) \, d\sigma(y) \\
&=
\int_{\mathbb{D}} \Phi_k(x,\tau(u)) e^{i \lambda \xi \cdot u} \chi_0(u) \mu(u)^2\,  du \\
&=
\frac{1}{4\pi}\mu(0)^2 \int_{\mathbb{R}^2} \frac{1}{|D\tau(0)u|} e^{i \lambda \xi \cdot u} \chi_0(u) \,  du + O_{\lambda \to \infty}(\lambda^{-2}) \\
&=
\frac{1}{4\pi}\mu(0) \int_{\mathbb{R}^2} \frac{1}{|u|} e^{i \lambda \xi \cdot u}\chi_0(u) \, du + O_{\lambda \to \infty}(\lambda^{-2}),
\end{align*}
and therefore $S_k \in \Psi^{-1}(\Gamma)$ with
\begin{align*}
\sigma^{-1}(S_k)(x,\xi) 
&=  
\frac{1}{4\pi}\mu(0)\lim_{\lambda \to \infty}  \int_{\mathbb{R}^2} \frac{1}{|u|} e^{i \xi \cdot u}\chi_0\Big(\frac{u}{\lambda}\Big) \, du \\
&=
\mu(0)\frac{1}{2|\xi|} \\
&=
\frac{1}{2|\xi|_g},
\end{align*}
where we have replaced the kernel by its singular part, the remainder is zero at $u=0$.
In fact, it is of type $O_{u\to 0}(|u|)$, which even gives a $O_{\lambda \to \infty}(\lambda^{-3})$ term by homogeneity.
Writing $b_{\alpha, \beta} = -\partial_\alpha \tau(0) \cdot \partial_\beta \nu(0)= \mathrm{II}(\partial_\alpha \tau(0), \partial_\beta \tau(0) )$, a similar calculation gives
\begin{align*}
e^{-i \lambda \varphi(x)} D_k( e^{i \lambda \varphi} \chi )(x)
&=
\int_{\mathbb{D}} \nu(\tau(u))\cdot \nabla_y\Phi_k(x,\tau(u)) e^{i \lambda \xi\cdot u} \chi_0(u) \mu(u)^2 \, du \\
&=
\frac{1}{8\pi\mu(0)}\sum_{\alpha,\beta=1}^2 b_{\alpha,\beta}\int_{\mathbb{R}^2} \frac{u^\alpha u^\beta}{|u|^3}  e^{i \lambda \xi \cdot u} \chi_0(u) \, du + O_{\lambda \to \infty}(\lambda^{-2}),
\end{align*}
and so $D_k \in \Psi^{-1}(\Gamma)$ with
\begin{align*}
\sigma^{-1}(D_k)(x,\xi) 
&=
\frac{1}{2\mu(0)}\sum_{\alpha,\beta=1}^2 b_{\alpha,\beta} \frac{1}{2|\xi|}\Big( \delta^{\alpha}_\beta - \frac{\xi_\alpha \xi_\beta}{|\xi|^2} \Big) \\
&=
\frac{1}{4|\xi|_g}\Big( 2H - \mathrm{II}\Big(\frac{\xi^\sharp}{|\xi|_g}, \frac{\xi^\sharp}{|\xi|_g}\Big) \Big),
\end{align*}
where $\mathrm{II}$ is the second fundamental form, and $H= \frac{1}{2}\textnormal{tr}_g\, \mathrm{II}$ is the mean curvature.

\newpage
Of course, the transpose $D_k^t$ has the same principal symbol $\sigma^{-1}(D_k^t)=\sigma^{-1}(D_k)$.
It is an elementary fact that the transposed operator has the same order as the orginal.
The magnetic operator $M_k$ is vectorial, and requires a slightly different computation.
Using the vector triple identity, we have for any $a\in C^\infty(\Gamma; T_\mathbb{C}\Gamma)$ that
\begin{align*}
    M_ka(x) = \int_\Gamma \nabla_x \Phi_k(x,y) (\nu(x) \cdot a(y)) \, d\sigma(y) - D_k^t a(x),
\end{align*}
where $D_k^t$ obviously acts component-wise, and we denote the former operator by $M_k'$.
Note that $\nu(x) \cdot a(y) = O_{y\to x}(|x-y|)$ ensures the integrand is only weakly singular.
It is therefore only necessary to get the asymptotics for this slightly simpler operator. 
In order to do this, we take $v\in T_\mathbb{C} \Gamma$ and define for $y\in \Gamma$ the tangential field
\begin{align*}
  a(y) =  -\nu(y) \times (\nu(x) \times v) \chi(y), 
\end{align*}
which is smooth, compactly supported in $\tau(\mathbb{D}^\circ)$, and has the property that $a(y) = v$.
Again, we proceed similarly to get
\begin{align*}
e^{-i \lambda \varphi(x)} M'_k( e^{i \lambda \varphi} a )(x)
&=  
\int_{\Gamma} \nabla_x \Phi_k(x,y) (- \nu(y)\cdot v) e^{i \lambda \varphi} \chi(y) \, d\sigma(y) \\
&=
\int_{\mathbb{D}} \nabla_x \Phi_k(x,\tau(u)) (- \nu(\tau(u))\cdot v) e^{i \lambda \xi\cdot u} \chi_0(u) \mu(u)^2  \ du \\
&=
m'_k(\lambda)(x,\xi)  + O_{\lambda \to \infty}(\lambda^{-2}),
\end{align*}
where
\begin{align*}
m'_k(\lambda)(x,\xi) 
&=
\frac{1}{4\pi} \mu(0)^2 \int_{\mathbb{R}^2} \frac{D\tau(0)u}{|D\tau(0)u|^3} \sum_{\beta=1}^2 u^\beta(- \partial_\beta\nu(0)\cdot v) e^{i \lambda \xi\cdot u} \chi_0(u) \, du  \\
&=
\frac{1}{4\pi \mu(0)} \sum_{\gamma=1}^2 \partial_{\gamma} \tau(0) \sum_{\alpha,\beta=1}^2 b_{\alpha,\beta} v^\alpha\int_{\mathbb{R}^2} \frac{u^\gamma u^\beta}{|u|^3}  e^{i \lambda \xi\cdot u} \chi_0(u) \, du,
\end{align*}
and this confirms the asymptotic corresponding to order $-1$, so $M_k \in \Psi^{-1}(\Gamma; T_\mathbb{C}\Gamma)$.
The first part of the principal symbol is then
\begin{align*}
\lim_{\lambda \to \infty} \lambda \,m'_k(\lambda)(x,\xi)
&=
\frac{1}{ \mu(0)}\sum_{\gamma=1}^2 \partial_{\gamma} \tau(0)\sum_{\alpha,\beta=1}^2 b_{\alpha,\beta} v^\alpha \frac{1}{2|\xi|}\Big( \delta^{\beta}_\gamma - \frac{\xi_\beta \xi_\gamma}{|\xi|^2} \Big) \\
&=
\frac{1}{2|\xi|_g} \Big( \mathcal{W}v - \mathrm{II}\Big(v, \frac{\xi^\sharp}{|\xi|_g}\Big) \frac{\xi^\sharp}{|\xi|_g}\Big),
\end{align*}
where $\mathcal{W} : T_\mathbb{C}\Gamma \to T_\mathbb{C}\Gamma$ is the Weingarten operator, a geometric vector bundle map.
Combining this with the previous calculation, we get
\begin{align*}
    \sigma^{-1}(M_k)(x,\xi)v 
    =
    \frac{1}{4|\xi|_g} \Big( 2(\mathcal{W}- H)v +  \mathrm{II}\Big(\frac{\xi^\sharp}{|\xi|_g}, \frac{\xi^\sharp}{|\xi|_g}\Big) v - 2\mathrm{II}\Big(\frac{\xi^\sharp}{|\xi|_g}, v\Big) \frac{\xi^\sharp}{|\xi|_g}\Big).
\end{align*}
The checks for the hyper-singular operator $W_k$, and $W_k-W_0$, are similar in spirit.
We leave it to the reader.
\end{proof}

\newpage

\bibliographystyle{siamplain}
\bibliography{references}
\end{document}